\documentclass[a4paper,11pt,reqno]{article}
\usepackage{latexsym,amsmath,amssymb,amsfonts,amsthm}
\usepackage{mathrsfs}
\usepackage{a4wide}
\usepackage[usenames, dvipsnames]{xcolor}
\usepackage{changes}
\usepackage[colorlinks=true, linkcolor=blue, citecolor=ForestGreen]{hyperref}

\newcounter{margcount}
\newcommand{\xupref}[2]{\hspace{-0.3ex}\stackrel{\eqref{#1}}{#2}} 

\begingroup
\newtheorem{theorem}{Theorem}[section]
\newtheorem{lemma}[theorem]{Lemma}
\newtheorem{proposition}[theorem]{Proposition}

\endgroup
\begingroup
\newtheorem{definition}[theorem]{Definition}
\endgroup
\newtheorem{remark}[theorem]{Remark}

\numberwithin{equation}{section}

\newcommand{\e}{\varepsilon}
\newcommand{\vphi}{\varphi}
\newcommand{\N}{\mathbb N}
\newcommand{\Z}{\mathbb Z}

\newcommand{\R}{\mathbb R}

\newcommand{\de}{\,\mathrm{d}}

\newcommand{\supp}{{\rm supp\,}}

\DeclareMathOperator*{\loc}{loc}

\newcommand{\measf}{\mathcal{M}_+((0,\infty))}
\newcommand{\measftheta}{\mathcal{M}_+((0,\infty);1+\xi^\theta)}
\newcommand{\measfbeta}{\mathcal{M}_+((0,\infty);1+\xi^{\beta+1})}
\newcommand{\measg}{\mathcal{M}_+(\R)}
\newcommand{\measgbeta}{\mathcal{M}_+(\R; 1 + 2^{(\beta+1) x})}
\newcommand{\cc}{C_{\mathrm c}}
\newcommand{\coag}{\mathscr{C}}
\newcommand{\frag}{\mathscr{F}}

\title{Solutions with peaks for a coagulation-fragmentation equation. Part I: stability of the tails}
\author{Marco Bonacini\thanks{\emailmarco} \and Barbara Niethammer\thanks{\emailbarbara} \and Juan J. L. Vel\'{a}zquez\thanks{\emailjuan}}
\date{${}^*$\addmarco \\ ${}^\dag$$^\ddag$\uniadd \\[3ex]\today}

\newcommand{\email}[1]{E-mail: \tt #1}
\newcommand{\emailmarco}{\email{marco.bonacini@unitn.it}}
\newcommand{\emailbarbara}{\email{niethammer@iam.uni-bonn.de}}
\newcommand{\emailjuan}{\email{velazquez@iam.uni-bonn.de}}
\newcommand{\addmarco}{\emph{\small Department of Mathematics, University of Trento\\ Via Sommarive 14, 38123 Povo (TN), Italy}}
\newcommand{\uniadd}{\emph{\small University of Bonn, Institute for Applied Mathematics\\ Endenicher Allee 60, 53115 Bonn, Germany}}

\begin{document}

\maketitle

\begin{abstract}
The aim of this two-part paper is to investigate the stability properties of a special class of solutions to a coagulation-fragmentation equation. We assume that the
coagulation kernel is close to the diagonal kernel, and that the fragmentation kernel is diagonal. We construct a two-parameter family of stationary solutions concentrated in Dirac masses. We carefully study the asymptotic decay of the tails of these solutions, showing that this behaviour is stable. In a companion paper we prove that for initial data which are sufficiently concentrated, 
the corresponding solutions approach one of these stationary solutions for large times.
\end{abstract}



\section{Introduction} \label{sect:intro}

The aim of this two-part paper is to investigate the stability properties of a special class of solutions to a coagulation-fragmentation model.
We consider the evolution equation
\begin{equation} \label{eq:coagfrag}
\begin{split}
\partial_t f(\xi,t) = \coag[f](\xi,t) + \frag[f](\xi,t)\,,
\end{split}
\end{equation}
where the coagulation operator and the fragmentation operator are respectively defined as
\begin{equation} \label{coag}
\coag[f](\xi,t) := \frac12\int_0^{\xi} K(\xi-\eta,\eta)f(\xi-\eta,t)f(\eta,t)\de \eta - \int_0^\infty K(\xi,\eta)f(\xi,t)f(\eta,t)\de \eta\,,
\end{equation}
\begin{equation} \label{frag}
\frag[f](\xi,t) := \int_0^\infty \Gamma(\xi+\eta,\eta)f(\xi+\eta,t)\de\eta - \frac12\int_0^\xi \Gamma(\xi,\eta)f(\xi,t)\de\eta\,.
\end{equation}
The mean-field equation \eqref{eq:coagfrag} describes the time evolution of a distribution $f(\xi,t)$ of particles of mass $\xi\geq0$ at time $t\geq0$. The reaction kernels $K$ and $\Gamma$ model the rate at which particles coalesce or fragment. The meaning of the different terms in the equation is the following: the first term in \eqref{coag} and the first term in \eqref{frag} account for the formation of particles of size $\xi$, by coalescence of particles of smaller sizes $\eta$ and $\xi-\eta$, or by fragmentation of particles of larger size $\xi+\eta$; the second term in \eqref{coag} and the second term in \eqref{frag} describe the depletion of particles of size $\xi$ by the reverse processes.
We refer the reader to the recent books \cite{BLL19a,BLL19b} and 
the references therein for an account of the general properties of equations in the form \eqref{eq:coagfrag} and their applications.

In the following we will assume that the coagulation kernel $K(\xi,\eta)$ is compactly supported around the diagonal $\{\xi=\eta\}$, while the fragmentation kernel 
$\Gamma(\xi,\eta)$ is purely diagonal (that is, a particle can only split in a pair of particles of the same size). The kernels will also satisfy suitable growth 
conditions at the origin and at infinity, see Section~\ref{sect:weaksol} for the precise assumptions. It turns out that for such kernels the equation \eqref{eq:coagfrag} has a
two-parameter family of stationary solutions with peaks concentrated in Dirac masses: more precisely, given any value of the total mass $M>0$ and a 
shifting parameter $\rho\in[0,1)$, there exists a (measure) solution to \eqref{eq:coagfrag} in the form
\begin{equation} \label{intro1}
f_p(\xi;M,\rho) = \sum_{n=-\infty}^\infty f_n(M,\rho)\delta(\xi-2^{n+\rho})
\end{equation}
with total mass $\int_0^\infty\xi f_p(\xi;M,\rho)\de\xi=\sum_{n=-\infty}^\infty 2^{n+\rho} f_n(M,\rho) = M$ (see Section~\ref{sect:stationary} for details). Thus,
there is in particular also nonuniqueness of steady states since there exists a one-parameter family of stationary solutions for any given mass.
It can be seen that the diagonal kernel is the only fragmentation kernel for which the system has stationary solutions with this type of behaviour, except for additions
to $\Gamma$ of measures supported in the complement of the lines $\{\xi+\eta=2^{n+\rho}\}$.
The existence of the equilibria \eqref{intro1} imply that the coagulation and fragmentation coefficients satisfy the so-called \emph{detailed balance condition}, 
that is the identity
\begin{equation*}
K(\xi,\eta)\mu(\xi)\mu(\eta) = \Gamma(\xi+\eta,\eta)\mu(\xi+\eta)
\end{equation*}
holds for the family of measures $\mu=f_p$. Notice, however, that this condition is satisfied in a slightly atypical sense, as the stationary states $f_p$ 
are mutually singular for different values of $\rho$. Another consequence of this nonuniqueness of stationary solutions and their
singular structure is that we have a corresponding family of entropies, but they are finite
only for measures that have the same support as the respetive stationary solution. Thus, we cannot  use an
$H$-Theorem to show for general data convergence to a steady state as has been done in \cite{LM03} (for the continuous case) and \cite{Can07} (in the discrete case).

It is the purpose of this paper to start a rigorous analysis of the properties of the stationary solutions \eqref{intro1}. Here we are mainly concerned with the stability of the \emph{tails}: the construction of the solution $f_p$ shows that the number of particles $f_n$ in the peak located at the point $\xi=2^{n+\rho}$ decays exponentially as $\xi\to\infty$, apart from a lower order algebraic correction; more precisely, there exists a constant $A_M>0$, uniquely determined by the total mass $M$ of the solution, such that
\begin{equation} \label{intro2}
f_n(M,\rho) \sim c\,2^{an}e^{-A_M2^n} \qquad\text{as } n\to\infty,
\end{equation}
where $c>0$ and $a>0$ are constants depending only on the properties of the kernels and on the shifting parameter $\rho$.

The main contribution of this paper is to show that this tail behaviour is stable, in the following sense. Consider an initial datum which is supported in the discrete set $\{2^{n+\rho}\}_{n\in\Z}$:
\begin{equation*}
h^0(\xi) = \sum_{n=-\infty}^\infty h_n^0\delta(\xi-2^{n+\rho}), \qquad\text{with } M=\int_0^\infty\xi h^0(\xi)\de\xi\,;
\end{equation*}
assume further that $h^0$ is a small perturbation of a stationary state (not necessarily the one with the same mass): $h_n^0 = (1+\e_n^0)f_n(M^0,\rho)$, for some $M^0$ sufficiently close to $M$ and coefficients $\e_n^0$ small enough. In our main result (Theorem~\ref{thm:dirac}) we show that this behaviour is persistent: we construct a solution $h(\xi,t)$ to \eqref{eq:coagfrag} starting from $h^0$ which remains concentrated in the points $\{2^{n+\rho}\}_{n\in\Z}$ and can be written in the form
\begin{equation*}
h(\xi,t) = \sum_{n=-\infty}^\infty (1+\e_n(t))f_n(M(t),\rho)\delta(\xi-2^{n+\rho})
\end{equation*}
for suitable coefficients $\e_n(t)$, exponentially decaying to zero as $t\to\infty$, and for a map $t\mapsto M(t)$ with $\lim_{t\to\infty}M(t)=M$. In other words, at each positive time $t$ the solution is a perturbation of one of the stationary states $f_p$; this is adjusted at each time by ``tuning'' the parameter $M(t)$ (and therefore the asymptotic decay \eqref{intro2} of the tail), and eventually approaches the stationary state corresponding to the initial mass $M$.

The existence of the map $t\mapsto M(t)$ is proved by means of a fixed point argument. A relevant part of the proof consists in the analysis of a discrete equation of the form
\begin{equation} \label{intro3}
\frac{\de y_n}{\de t} = 2^{\beta n} \Bigl[ y_{n-1}-y_n - \sigma_n\bigl( y_n-y_{n+1}\bigr) \Bigr],
\qquad\qquad n\in\Z\quad (\beta>1),
\end{equation}
which appears as the linearization of the evolution equation for the coefficients of $h(\xi,t)$. Hence the study of this problem requires the development of a well-posedness theory for \eqref{intro3} and the detailed characterization of the asymptotic behaviour of solutions to \eqref{intro3} as $n\to\infty$. In a sense, these results can be thought as the analogue of the H\"older regularity theory for parabolic equations, in this discrete setting and at the single point $n=\infty$.

In a companion paper \cite{BNVe} we will actually prove the stability of the solutions $f_p$: starting from an initial datum which is supported in the union of infinitely many small intervals around the points $\xi_n=2^{n+\rho}$, we show the existence of a corresponding solution to \eqref{eq:coagfrag} which actually converges to the stationary solution \eqref{intro1} with the same mass. This results from the combination of two main effects: first, the phenomenon of the stability of the exponential decay of the tails, discussed in this paper; second, a tendency to aggregation leading to concentration into peaks, which can be quantified in terms of a quadratic functional representing some kind of variance. We refer to the introduction of \cite{BNVe} for more details.

The careful analysis of the tails of the solutions performed in this paper, besides being interesting by itself, is instrumental in the proof of the full result contained in \cite{BNVe}. In particular the study of the linearized system \eqref{intro3} plays a fundamental role. Finally, one of the goals of this paper is also to give an insight of the strategy of the proof of the result in \cite{BNVe}, free of the technical details due to the presence of the dispersion around the peaks.

Our motivation of the present study actually arose from  an investigation of self-similar solutions to pure coagulation equations. While it was generally expected that
for kernels with homogeneity smaller than or equal to one 
solutions converge to a uniquely determined self-similar profile, proofs of corresponding results are only available for the constant and the additive kernel 
(see \cite{MP04} and references therein). In a study based on formal stability arguments and numerical simulations \cite{HNV16} it has however been noticed that for kernels that concentrate
near the diagonal solutions tend to aggregate in peaks. 
We expect that the analysis developed in this paper can also  be extended  to the pure coagulation equation, for coagulation  kernels as considered here.
Indeed, a family of solutions supported for each time in a set of Dirac masses, similar to \eqref{intro1}, can be constructed 
also for the pure coagulation equation; in this case the solutions, in self-similar variables, are not stationary, but periodic in time.
 One of the main difficulties consists then in proving a precise characterization of the tails, as in \eqref{intro2}. This would allow to repeat 
 the strategy devised here and in \cite{BNVe} and to prove the stability of such solutions. When the kernel is purely diagonal, a related result has been 
 established in \cite{LNV18}.

\medskip
\noindent\textbf{Structure of the paper.} The paper is organized as follows. In Section~\ref{sect:weaksol} we formulate the precise assumptions on the coagulation and fragmentation kernels, and we prove a well-posedness result for (a suitable weak formulation of) the equation \eqref{eq:coagfrag}. In Section~\ref{sect:variables} we reformulate the equation by means of a convenient change of variable. The family of stationary solutions in the form \eqref{intro1} is constructed in Section~\ref{sect:stationary}. In Section~\ref{sect:linear} we state the regularity result on the linearized equation \eqref{intro3}, whose proof is postponed to Appendix~\ref{sect:appendix}. Finally, in Section~\ref{sect:stability1} we prove the main result of this paper (Theorem~\ref{thm:dirac}), which shows the stability of the stationary solutions among the class of solutions concentrated in peaks.


\section{Assumptions, weak formulation and well-posedness} \label{sect:weaksol}

We now formulate the precise assumptions on the coagulation and fragmentation kernels, that are valid throughout the paper. Afterwards we define the notion of weak solution to \eqref{eq:coagfrag}, and we prove a standard existence result.

We assume the following: the coagulation kernel is supported near the diagonal and has the form
\begin{equation} \label{kernel1}
K(\xi,\eta) = \frac{1}{\xi+\eta}k\Bigl(\frac{\xi+\eta}{2}\Bigr) Q\Bigl(\frac{2\eta}{\xi+\eta}-1\Bigr)\,,
\end{equation}
where $k\in C([0,\infty))$, $k>0$, satisfies the growth conditions
\begin{align}
\qquad\qquad\qquad&k(\xi)\sim\xi^{\alpha+1} & &\text{as }\xi\to\infty, \quad\alpha\in(0,1), & \qquad\qquad\qquad \label{kernel2}\\
\qquad\qquad\qquad&k(\xi)= k_0 + O(\xi^{\bar{\alpha}}) & &\text{as }\xi\to0^+, \quad\bar{\alpha}>1, & \qquad\qquad\qquad \label{kernel2bis}
\end{align}
for some $k_0>0$, and $Q$ is a cut-off function such that
\begin{equation} \label{kernel3}
Q\in C(\R), \quad Q\geq0, \quad Q(0)=1, \quad \supp Q\subset\bigl(-{\textstyle\frac13},{\textstyle\frac13}\bigr), \quad Q(\xi)=Q(-\xi).
\end{equation}
The kernel $K$ has been written in the form \eqref{kernel1} to emphasize that it is close to the diagonal kernel in the sense of measures: indeed if we take $Q(s)=\delta(s)$ (where $\delta$ denotes the Dirac distribution at the origin) only particles with the same size coagulate, and we recover the diagonal kernel $K(\xi,\eta)=k(\xi)\delta(\xi-\eta)$. The exponent $\alpha$ represents the degree of homogeneity at infinity. The condition on the support of $Q$ guarantees that, for solutions concentrated in Dirac masses at points $\{2^n\}_{n\in\Z}$, the different peaks do not interact with each other; in particular
\begin{equation} \label{suppK0}
\supp K(\xi,\eta) \subset \Bigl\{ \frac12\xi < \eta < 2\xi \Bigr\}.
\end{equation}

As for the fragmentation kernel $\Gamma(\xi,\eta)$, we assume that
\begin{equation} \label{kernel4}
\Gamma(\xi,\eta) = \gamma(\xi)\delta(\xi-2\eta)\,,
\end{equation}
where $\gamma\in C(0,\infty)$, $\gamma(\xi)>0$, satisfies the growth conditions
\begin{align}
\qquad\qquad\qquad&\gamma(\xi)\sim\xi^{\beta} & &\text{as }\xi\to\infty, \quad\beta\in(1,2), & \qquad\qquad\qquad \label{kernel5}\\
\qquad\qquad\qquad&\gamma(\xi)= \gamma_0 + O(\xi^{\bar{\beta}}) & &\text{as }\xi\to0^+, \quad\bar{\beta}>1, & \qquad\qquad\qquad \label{kernel5bis}
\end{align}
for some $\gamma_0>0$.

\begin{remark} \label{rm:ass}
Some of the assumptions above (for instance, the condition $\beta<2$, see Remark~\ref{rm:ass2}) are of technical nature and could be probably removed. However, since our ultimate goal is to provide an example of solutions concentrating in peaks, we prefer to make stronger assumptions and keep the amount of technical details of the paper within reasonable limits.
\end{remark}

\subsection{Weak formulation and well-posedness}
We introduce a weak formulation of the equation in the space of positive Radon measures $f\in\measf$, that allows to consider solutions to \eqref{eq:coagfrag} concentrated in Dirac masses and to prove a general well-posedness result. In the following, with abuse of notation, we denote by $\int_A\phi(\xi)f(\xi)\de\xi$ the integral of $\phi$ on $A\subset(0,\infty)$ with respect to the measure $f$, also in the case that $f$ is not absolutely continuous with respect to the Lebesgue measure.
It will be also convenient to consider the space $\measftheta$ of positive measures $f\in\measf$ such that $\int_{(0,\infty)} (1+\xi^\theta)f(\xi)\de\xi<\infty$.

In order to formally derive the weak formulation of the equation, we multiply \eqref{eq:coagfrag} by a test function $\vphi(\xi)$ and we integrate over $(0,\infty)$:
\begin{equation*}
\partial_t\biggl(\int_0^\infty f(\xi,t)\vphi(\xi)\de\xi \biggr) = \int_{0}^\infty \coag[f](\xi,t)\vphi(\xi)\de\xi + \int_{0}^\infty \frag[f](\xi,t)\vphi(\xi)\de\xi \,.
\end{equation*}
For the coagulation part we have, using Fubini's Theorem,
\begin{align*}
\int_{0}^\infty \coag[f](\xi,t)\vphi(\xi)\de\xi
& = \frac{1}{2}\int_{0}^\infty\de\eta \int_{\eta}^{\infty} K(\xi-\eta,\eta) f(\xi-\eta,t)f(\eta,t)\vphi(\xi)\de\xi \\
&\qquad - \int_{0}^\infty\int_{0}^\infty K(\xi,\eta)f(\xi,t)f(\eta,t)\vphi(\xi)\de\xi\de\eta\\
& = \frac{1}{2}\int_{0}^\infty\int_{0}^\infty K(\xi,\eta)f(\xi,t)f(\eta,t)\Bigl(\vphi(\xi+\eta)-\vphi(\xi)-\vphi(\eta)\Bigr)\de\xi\de\eta \,.
\end{align*}
A similar computation for the fragmentation term, taking into account the diagonal structure \eqref{kernel4} of the kernel $\Gamma(\xi,\eta)$, shows that
\begin{align*}
\int_{0}^\infty \frag[f](\xi,t)\vphi(\xi)\de\xi
& = - \frac{1}{2}\int_{0}^\infty\int_{0}^\infty \Gamma(\xi+\eta,\eta)f(\xi+\eta,t)\Bigl(\vphi(\xi+\eta)-\vphi(\xi)-\vphi(\eta)\Bigr)\de\xi\de\eta \\
& = -\frac12\int_{0}^\infty\int_{0}^\infty \gamma(2\xi)f(2\xi,t)\Bigl(\vphi(\xi+\eta)-\vphi(\xi)-\vphi(\eta)\Bigr)\delta(\xi-\eta)\de\xi\de\eta \,.
\end{align*}
These formal identities motivate the following definition.

\begin{definition}[Weak solution] \label{def:weakf}
A map $f\in C([0,T];\measfbeta)$ is a \emph{weak solution} to \eqref{eq:coagfrag} in $[0,T]$ with initial condition $f_0\in\measf$ if for every $t\in[0,T]$
\begin{align} \label{weakf}
&\partial_t \biggl(\int_{(0,\infty)} f(\xi,t)\vphi(\xi)\de\xi \biggr) \\
&=\frac{1}{2}\int_{0}^\infty\int_{0}^\infty \Bigl[ K(\xi,\eta)f(\xi,t)f(\eta,t) - \gamma(2\xi)f(2\xi,t)\delta(\xi-\eta)\Bigr]\Bigl(\vphi(\xi+\eta)-\vphi(\xi)-\vphi(\eta)\Bigr)\de\xi\de\eta \nonumber
\end{align}
for every test function $\vphi\in C([0,\infty))$ with $\vphi(\xi)\lesssim\xi$ as $\xi\to\infty$, and $f(\cdot,0)=f_0$.
\end{definition}

Notice that this notion of weak solution guarantees the mass conservation property
\begin{equation} \label{massf}
\int_{(0,\infty)}\xi f(\xi,t)\de\xi = \int_{(0,\infty)}\xi f_0(\xi)\de\xi \qquad\text{for all $t>0$},
\end{equation}
as follows immediately by choosing the admissible test function $\vphi(\xi)=\xi$.  This is the reason why we require that the identity \eqref{weakf} holds for test functions with linear growth at infinity. In order for the fragmentation term on the right-hand side to be well-defined for such test functions, in view of the growth assumption \eqref{kernel5} we also need to assume the boundedness of the moment of order $\beta+1$ of the solution. The existence of a (global in time) weak solution for a suitable class of initial data is proved in Theorem~\ref{thm:wp} below.

A well-posedness result for the weak formulation of the coagulation-fragmentation equation \eqref{eq:coagfrag}, with a continuous coagulation kernel and a fragmentation measure, is proved for instance in \cite{EibWag00}; unfortunately we cannot refer directly to that result, as it requires some restriction on the growth of the coagulation and fragmentation rates, which do not apply to our setting. In order to show that a weak solution exists, it is convenient to introduce the following norm on the space of positive Radon measures $\measf$:
\begin{equation} \label{normf}
\|f\| := \sup_{\substack{n\in\Z\\n< 0}} \, \frac{1}{2^n}\int_{[2^n,2^{n+1})}f(\xi)\de\xi + \int_{[1,\infty)} f(\xi)\de\xi\,, \qquad f\in\measf.
\end{equation}
The reason for introducing this norm is that (taking into account the form of the steady states, see Section~\ref{sect:stationary}) we want to consider only solutions that are bounded as $\xi\to0^+$. Notice that the finiteness of the norm \eqref{normf} implies
\begin{equation} \label{normf2}
\int_{(0,\infty)} f(\xi)\de\xi = \sum_{n=-\infty}^{-1}\int_{[2^n,2^{n+1})}f(\xi)\de\xi + \int_{[1,\infty)}f(\xi)\de\xi \leq 2\|f\|.
\end{equation}
The existence of weak solutions with finite norm $\|\cdot\|$ and conserved mass, for a given initial datum, is guaranteed by the following theorem.

\begin{theorem}[Existence of weak solutions] \label{thm:wp}
Suppose that $f_0\in\measf$ satisfies
\begin{equation} \label{moment0}
\|f_0\|<\infty, \qquad \int_{(0,\infty)} \xi^\theta f_0(\xi)\de\xi <\infty
\end{equation}
for some $\theta>\beta+1$. Then there exists a weak solution $f$ to \eqref{eq:coagfrag} in $[0,\infty)$ with initial datum $f_0$, according to Definition~\ref{def:weakf}, which satisfies for all $T>0$
\begin{equation} \label{wpestf}
\sup_{0\leq t\leq T}\|f(t)\| \leq C(T, f_0), \qquad
\sup_{0\leq t\leq T}\int_{(0,\infty)} \xi^\theta f(\xi,t)\de\xi \leq C(T, f_0),
\end{equation}
where $C(T, f_0)$ denotes a constant depending on $T$, $f_0$, and on the properties of the kernels.
\end{theorem}

\begin{proof}
In oder to prove the result it is convenient to truncate the kernels and write the equation in integral form; in this way the existence of a (strong) solution can be obtained by a standard fixed point argument. For any given $R>1$, we consider a cut-off function $\psi_R\in \cc^1([0,R))$, $\psi_R\equiv1$ in $[0,R-1]$, and we define
\begin{equation} \label{truncation}
K_R(\xi,\eta) := K(\xi,\eta)\psi_R(\xi),
\qquad
\Gamma_R(\xi,\eta) := \Gamma(\xi,\eta)\psi_R(\xi),
\end{equation}
and we notice that \eqref{eq:coagfrag} (with the truncated kernels) can be rewritten as
\begin{equation} \label{truncation2}
\partial_tf(\xi,t) + A_R[f](\xi,t)f(\xi,t) = B_R[f](\xi,t)
\end{equation}
where (recall \eqref{suppK0} and \eqref{kernel4})
\begin{equation} \label{truncation3}
A_R[f](\xi,t)
:= \biggl(\int_{\frac12\xi}^{2\xi} K(\xi,\eta)f(\eta,t)\de \eta + \frac14\gamma(\xi) \biggr) \psi_R(\xi)\,,
\end{equation}
\begin{equation} \label{truncation4}
B_R[f](\xi,t)
:= \frac12\int_{\frac13\xi}^{\frac23\xi} K(\xi-\eta,\eta)\psi_R(\xi-\eta)f(\xi-\eta,t)f(\eta,t)\de \eta + \gamma(2\xi)f(2\xi,t)\psi_R(2\xi) \,.
\end{equation}

\smallskip\noindent\textit{Step 1.} Suppose that $f_0\in\measf$ satisfies $\|f_0\|<\infty$. We show that there exist a time $T=T(R)>0$ and a map $f_R\in C([0,T];\measf)$ which obeys for every $t\in[0,T]$ the identity
\begin{multline} \label{smolmild}
f_R(\xi,t) = f_0(\xi)\exp\biggl(-\int_0^t A_R[f_R](\xi,s)\de s \biggr) \\
+ \int_0^t \exp\biggl(-\int_s^t A_R[f_R](\xi,r)\de r \biggr) B_R[f_R](\xi,s)\de s =: \mathcal{T}_R[f_R](\xi,t)
\end{multline}
in the sense of measures. The goal is to show that the operator $\mathcal{T}_R$ maps the set
\begin{equation*}
\mathcal{U} := \Bigl\{ f\in C([0,T];\measf) \,:\, \sup_{t\in[0,T]}\|f(t)\| \leq 2\|f_0\| \Bigr\}
\end{equation*}
into itself, and is strongly contractive if $T$ is sufficiently small (depending on $R$).

Let $f\in\mathcal{U}$. We first estimate the term $A_R[f]$: for $\xi\in[0,1]$ we have, by \eqref{kernel1}, 
\begin{equation*}
\begin{split}
A_R[f](\xi,t)
& \leq  \int_{\frac12\xi}^{2\xi}K(\xi,\eta)f(\eta,t)\de\eta + \frac14\bigl(\max_{\xi\in[0,1]}\gamma(\xi)\bigr) \\
& \leq C \biggl( \frac{1}{\xi}\int_{\frac12\xi}^{2\xi}f(\eta,t)\de\eta + 1 \biggr)
\leq C \bigl( \|f(\cdot,t)\| + 1 \bigr)
\end{split}
\end{equation*}
(where $C$ is a positive constant depending only on the kernels), while for $\xi\geq1$ it is easily seen (recall \eqref{normf2}) that
\begin{equation*}
A_R[f](\xi,t) \leq \bigl(\max_{\frac12\leq\xi,\eta\leq 2R}K(\xi,\eta)\bigr)\int_{(0,\infty)}f(\eta,t)\de\eta +  \frac14\bigl(\max_{\xi\in[0,R]}\gamma(\xi)\bigr)
\leq C_R\bigl( \|f(\cdot,t)\|+1\bigr) \,.
\end{equation*}
Combining the previous bounds we therefore obtain
\begin{equation} \label{truncation5}
A_R[f](\xi,t) \leq C_R \bigl(\|f_0\|+1) \,.
\end{equation}
We next estimate the term $B_R[f]$ in the norm $\|\cdot\|$: for all $n\in\Z$, $n<0$, we have
\begin{equation*}
\begin{split}
\frac{1}{2^n}\int_{[2^n,2^{n+1})} B_R[f](\xi,t)\de\xi
&\leq \frac{C}{2^{n+1}}\int_{[2^n,2^{n+1})}\frac{\de\xi}{\xi}\int_{[\frac13\xi,\frac23\xi]} f(\xi-\eta,t)f(\eta,t)\de\eta \\
& \qquad + \bigl(\max_{\xi\in[0,2]}\gamma(\xi)\bigr) \frac{1}{2^{n+1}} \int_{[2^{n+1},2^{n+2})}f(\xi,t)\de\xi \\
&\leq \frac{C}{2^{n+1}}\int_{[2^{n-2},2^{n+1}]}\frac{\de\xi}{2^n}\int_{[2^{n-2},\xi)} f(\xi-\eta,t)f(\eta,t)\de\eta + C\|f(\cdot,t)\| \\
&\leq C\biggl(\frac{1}{2^{n+1}}\int_{[2^{n-2},2^{n+1}]}f(\eta,t)\de\eta \biggr) \biggl(\frac{1}{2^n} \int_{(0,2^{n+1}]} f(\zeta,t)\de\zeta\biggr) + C\|f_0\| \\
& \leq C\|f(\cdot,t)\|^2 + C\|f_0\| \leq C \bigl(\|f_0\|^2+\|f_0\|\bigr),
\end{split}
\end{equation*}
while, recalling once more \eqref{normf2},
\begin{equation*}
\begin{split}
\int_{[1,\infty)} B_R[f](\xi,t)\de\xi
&\leq \frac12\bigl(\max_{\frac13\leq\xi,\eta\leq 2R}K(\xi,\eta)\bigr)\int_{[1,\infty)}\de\xi\int_{\frac13\xi}^{\frac23\xi} f(\xi-\eta,t)f(\eta,t)\de\eta \\
& \qquad + \bigl(\max_{\xi\in[0,R]}\gamma(\xi)\bigr)\int_{[1,\infty)}f(\xi,t)\de\xi \\
&\leq C_R \biggl( \int_{[\frac13,\infty)}f(\eta,t)\de\eta\int_{(0,\infty)}f(\zeta,t)\de\zeta + \|f(\cdot,t)\| \biggr) \\
& \leq C_R \bigl( \|f(\cdot,t)\|^2 + \|f(\cdot,t)\| \bigr) \leq C_R \bigl(\|f_0\|^2+\|f_0\|\bigr) .
\end{split}
\end{equation*}
By combining the previous estimates we obtain
\begin{equation} \label{truncation6}
\|B_R[f](\cdot,t)\| \leq C_R (\|f_0\|^2+\|f_0\|)
\end{equation}
for some positive constant $C_R$ depending on $R$.

It follows that, for $f\in\mathcal{U}$, $A_R[f](\cdot,t)$ is continuous and bounded, and $B_R[f]$ is a bounded, positive measure (it is a weighted convolution of measures); hence $\mathcal{T}_R[f]$ is well-defined. 
Moreover, by combining \eqref{truncation5} and \eqref{truncation6} we find
\begin{equation*}
\sup_{t\in[0,T]} \| \mathcal{T}_R[f](\cdot,t) \| \leq \|f_0\| + C_R (\|f_0\|^2+\|f_0\|)T,
\end{equation*}
which shows that $\mathcal{T}_R[f]\in\mathcal{U}$ for every $f\in\mathcal{U}$, provided that $T=T(R)$ is chosen small enough.

A similar argument also shows that for all $f_1,f_2\in\mathcal{U}$
\begin{equation*}
\big| A_R[f_1](\xi,t) - A_R[f_2](\xi,t) \big| \leq C_R \| f_1(\cdot,t)-f_2(\cdot,t) \|,
\end{equation*}
\begin{equation*}
\big\| B_R[f_1](\cdot,t) - B_R[f_2](\cdot,t) \big\| \leq C_R \| f_1(\cdot,t)-f_2(\cdot,t) \|,
\end{equation*}
from which it follows that
\begin{equation*}
\sup_{t\in[0,T]}\| \mathcal{T}_R[f_1](\cdot,t) - \mathcal{T}_R[f_2](\cdot,t) \| \leq C_R T \sup_{t\in[0,T]}\|f_1(\cdot,t)-f_2(\cdot,t)\|.
\end{equation*}
Therefore $\mathcal{T}_R$ is strongly contractive in $\mathcal{U}$ if $T$ is sufficiently small (depending on $R$ and on the initial datum), and Banach's fixed point theorem yields the existence of a unique solution $f_R$ in $\mathcal{U}$ to the equation \eqref{smolmild} in $[0,T]$.

\smallskip\noindent\textit{Step 2.}
We now observe that $f_R$ is a weak solution to \eqref{truncation2}. Indeed we can test \eqref{smolmild} against any compactly supported test function $\vphi\in \cc([0,\infty))$:
\begin{multline*}
\int_{(0,\infty)}\vphi(\xi)f_R(\xi,t)\de\xi
= \int_{(0,\infty)}\exp\biggl(-\int_0^t A_R[f_R](\xi,s)\de s \biggr)\vphi(\xi)f_0(\xi)\de\xi \\
+ \int_0^t\int_{(0,\infty)}\exp\biggl(-\int_s^t A_R[f_R](\xi,r)\de r \biggr) \vphi(\xi) B_R[f_R](\xi,s)\de\xi \de s \,.
\end{multline*}
By taking $\psi\in C^1([0,T])$, multiplying the previous equation by $\partial_t{\psi}$, and integrating in $[0,T]$, we obtain, after integration by parts,
\begin{multline*}
\psi(T)\int_{(0,\infty)}\vphi(\xi)f_R(\xi,T)\de\xi - \psi(0)\int_{(0,\infty)}\vphi(\xi)f_0(\xi)\de\xi - \int_0^T \partial_t\psi\biggl(\int_{(0,\infty)}\vphi(\xi)f_R(\xi,t)\de\xi\biggr)\de t \\
= -\int_0^T \psi(t)\biggl( \int_{(0,\infty)}A_R[f_R](\xi,t)\vphi(\xi)f_R(\xi,t)\de\xi - \int_{(0,\infty)}\vphi(\xi)B_R[f_R](\xi,t)\de\xi \biggr)\de t\,,
\end{multline*}
which is the weak formulation of \eqref{truncation2}. Then, passages similar to those preceding Definition~\ref{def:weakf} give
\begin{equation} \label{weakf2}
\begin{split}
\int_{(0,\infty)} & \vphi(\xi)f_R(\xi,T)\de\xi - \int_{(0,\infty)} \vphi(\xi)f_0(\xi)\de\xi \\
&=\frac{1}{2}\int_0^T\int_0^\infty\int_0^\infty K_R(\xi,\eta)f_R(\xi,t)f_R(\eta,t)\Bigl(\vphi(\xi+\eta)-\vphi(\xi)-\vphi(\eta)\Bigr)\de\xi\de\eta\de t \\
&\qquad - \frac12\int_0^T\int_0^\infty \gamma(2\xi)\psi_R(2\xi)f_R(2\xi,t)\Bigl(\vphi(2\xi)-2\vphi(\xi)\Bigr)\de\xi\de t .
\end{split}
\end{equation}
Observe also that, by an approximation argument, the identity \eqref{weakf2} actually holds for all test functions $\vphi\in C([0,\infty))$, not necessarily compactly supported, growing at most as $\vphi(\xi)\lesssim \xi^{\theta}$ as $\xi\to\infty$: indeed, the integral $\int_{(0,\infty)}f_0(\xi)\vphi(\xi)\de\xi$ is finite for such test functions, in view of the assumption \eqref{moment0}, and also the right-hand side of \eqref{weakf2} is well-defined, as the truncated kernels are compactly supported.

\smallskip\noindent\textit{Step 3.} 
We can now use the weak formulation \eqref{weakf2} to show that we can extend the solution $f_R$ to every positive time $t$, and to obtain uniform estimates on the moments
\begin{equation} \label{moments}
M_{r}(f_R(t)):=\int_{(0,\infty)}\xi^r f_R(\xi,t)\de\xi
\end{equation}
for $r\leq\theta$. Notice first that, by taking $\vphi(\xi)=\xi$ as test function in \eqref{weakf2}, we immediately obtain that the total mass of the solution is conserved as long as the solution exists: $M_1(f_R(t))\equiv M_1(f_0)$, where the initial mass $M_1(f_0)$ is finite in view of the assumptions \eqref{moment0}.

The local solution $f_R$ can be extended in time as long as an estimate of the form
\begin{equation} \label{truncation7}
\sup_{t\in[0,T]} \|f_R(\cdot,t)\| \leq C(R,T)
\end{equation}
holds. Since the constant function $\vphi\equiv1$ is admissible in \eqref{weakf2}, we have (by positivity of the kernels and of $f_R$)
\begin{equation*}
\begin{split}
\int_{(0,\infty)} f_R(\xi,T)\de\xi
& \leq \int_{(0,\infty)} f_0(\xi)\de\xi +\frac12 \int_0^T\int_{(0,\infty)}\gamma(2\xi)\psi_R(2\xi)f_R(2\xi,t)\de\xi\de t \\
& \leq \int_{(0,\infty)} f_0(\xi)\de\xi +\frac14\bigl(\max_{\xi\in[0,R]}\gamma(\xi)\bigr) \int_0^T\int_{(0,\infty)} f_R(\xi,t)\de\xi\de t,
\end{split}
\end{equation*}
and by a Gr\"onwall argument we obtain
\begin{equation} \label{truncation8}
\sup_{t\in[0,T]}\int_{(0,\infty)}f_R(\xi,t)\de\xi \leq C(R,T)\int_{(0,\infty)}f_0(\xi)\de\xi\,.
\end{equation}

By taking a sequence of test functions approaching the characteristic function of the interval $[0,2^{n}]$ ($n\in\Z$, $n\leq0$), and observing that for such functions the coagulation term gives a negative contribution, we find the estimate
\begin{equation*}
\begin{split}
\frac{1}{2^n}\int_{[0,2^{n}]} f_R(\xi,T)\de\xi
& \leq \frac{1}{2^n}\int_{[0,2^{n}]} f_0(\xi)\de\xi + \int_0^T \frac{1}{2^n}\int_{[0,2^{n}]}\gamma(2\xi)f_R(2\xi,t)\de\xi\de t \\
& \leq \|f_0\| + \bigl(\max_{\xi\in[0,2]}\gamma(\xi)\bigr) \int_0^T \frac{1}{2^{n+1}}\int_{[0,2^{n+1}]} f_R(\xi,t)\de\xi\de t \\
& \leq \|f_0\| + C\int_0^T \sup_{\substack{k\in\Z\\k\leq 0}} \frac{1}{2^{k}}\int_{[0,2^{k}]} f_R(\xi,t)\de\xi\de t + C\int_0^T\int_{[1,2]}f_R(\xi,t)\de\xi\de t \\
& \leq \|f_0\| + C\int_0^T \sup_{\substack{k\in\Z\\k\leq 0}} \frac{1}{2^{k}}\int_{[0,2^{k}]} f_R(\xi,t)\de\xi\de t + C M_1(f_0)T,
\end{split}
\end{equation*}
where $C$ is a positive constant, depending only on $\gamma$, and we used the conservation of mass of the solution in the last passage. Therefore a Gr\"onwall argument yields
\begin{equation} \label{truncation9}
\sup_{t\in[0,T]}\sup_{\substack{n\in\Z\\n\leq 0}}\frac{1}{2^n}\int_{[0,2^{n}]}f_R(\xi,t)\de\xi \leq C(T) \bigl(\|f_0\| + M_1(f_0) \bigr),
\end{equation}
where the constant $C(T)$ is now independent of $R$ (it depends only on $T$ and on the maximum of $\gamma$ in the interval $[0,2]$).
The combination of \eqref{truncation8} and \eqref{truncation9} implies in particular the estimate \eqref{truncation7} and, in turn, the global existence of the solution $f_R$.

We now consider the moment $M_\theta$. The (admissible) choice $\vphi(\xi)=\xi^\theta$ as test function in \eqref{weakf2}, using the elementary inequalities
\begin{equation*}
\vphi(\xi+\eta)-\vphi(\xi)-\vphi(\eta) \leq (2^{\theta-1}-1)(\xi^\theta+\eta^\theta), \qquad \vphi(2\xi)-2\vphi(\xi) \geq 0,
\end{equation*}
the assumption \eqref{suppK0} on the support of the coagulation kernel, the conservation of mass, and the estimate \eqref{truncation9}, gives
\begin{align*}
M_{\theta}(f_R(T)) - M_{\theta}(f_0)
& \leq (2^{\theta-1}-1) \int_0^T \de t\int_0^\infty\de\xi \int_{\frac12\xi}^{2\xi} K(\xi,\eta)f_R(\xi,t)f_R(\eta,t)\xi^\theta\de\eta \\
& \leq C\int_0^T\de t \int_0^1 \de\xi \, f_R(\xi,t)\xi^{\theta-1} \int_{\frac12\xi}^{2\xi} f_R(\eta,t)\de\eta \\
& \qquad + C\int_0^T\de t \int_1^{\infty} \de\xi \, f_R(\xi,t)\xi^{\theta} \int_{\frac12\xi}^{2\xi} \eta^\alpha f_R(\eta,t)\de\eta \\
& \leq C(T,f_0)\int_0^T\de t \int_0^1 f_R(\xi,t)\xi^{\theta}\de\xi +  C M_1(f_0) \int_0^T \de t \int_1^{\infty} f_R(\xi,t)\xi^\theta \de\xi \\
& \leq C(T,f_0) \int_0^T M_\theta(f_R(t))\de t
\end{align*}
for some constant $C(T,f_0)$ depending on the kernels, on $T$, and on the initial datum $f_0$ (independent of $R$). Hence we obtain the uniform estimate
\begin{equation} \label{truncation10}
\sup_{t\in[0,T]} M_\theta(f_R(t)) \leq C(T,f_0) \,,
\end{equation}
and in turn, by \eqref{truncation9} and \eqref{truncation10},
\begin{equation} \label{truncation11}
\sup_{t\in[0,T]} \|f_R(\cdot,t)\| \leq C(T,f_0) \,.
\end{equation}

\smallskip\noindent\textit{Step 4.} We finally obtain a weak solution to \eqref{eq:coagfrag} by passing to the limit as $R\to\infty$. We give only a sketch of the proof, which is relatively standard (a similar argument is used for instance in \cite[Section~4]{LauMis02}).

We first show the existence of a subsequence $R_j\to\infty$ along which the measures $f_{R_j}(\cdot,t)$ narrowly converge to some limit measure $f(\cdot,t)$ for every $t>0$, that is,
\begin{equation} \label{truncation12}
\int_{(0,\infty)} \vphi(\xi) f_{R_j}(\xi,t)\de\xi \to \int_{(0,\infty)} \vphi(\xi) f(\xi,t)\de\xi \qquad \text{for every $\vphi\in C_{\mathrm b}(0,\infty)$.}
\end{equation}
Recall that the narrow convergence is induced by a distance on the space of Borel probability measures (see \cite[Remark~5.1.1]{AGS}). Notice that, for fixed $t$, thanks to the uniform bound on the moments \eqref{truncation10}, the family of measures $\{ f_R(\cdot,t) \}_R$ is tight and hence relatively compact with respect to narrow convergence. We also need to show equicontinuity of the family $\{ f_{R} \}_R$: by \eqref{weakf2} we have, for all $\vphi\in C_{\mathrm b}((0,\infty))$ and all $0<s<t\leq T$,
\begin{equation*}
\begin{split}
\bigg| & \int_{(0,\infty)} \vphi(\xi) \bigl( f_{R}(\xi,t)-f_{R}(\xi,s) \bigr) \de\xi \bigg| \\
& \leq \bigg| \frac{1}{2}\int_s^t\int_0^\infty\int_0^\infty K_{R}(\xi,\eta)f_{R}(\xi,\tau)f_{R}(\eta,\tau)\Bigl(\vphi(\xi+\eta)-\vphi(\xi)-\vphi(\eta)\Bigr)\de\xi\de\eta\de\tau\\
&\qquad - \frac12\int_s^t\int_0^\infty \gamma(2\xi)\psi_{R}(2\xi)f_{R}(2\xi,\tau)\Bigl(\vphi(2\xi)-2\vphi(\xi)\Bigr)\de\xi\de\tau \bigg| \\
& \leq C\|\vphi\|_\infty \int_s^t \biggl(\int_0^1 \de\xi \, f_{R}(\xi,\tau)\frac{1}{\xi}\int_{\frac12\xi}^{2\xi} f_{R}(\eta,\tau)\de\eta + \int_1^{\infty} \de\xi \, f_{R}(\xi,\tau) \xi^\alpha \int_{\frac12\xi}^{2\xi} f_{R}(\eta,\tau)\de\eta\\
& \qquad + \int_0^1 f_{R}(\xi,\tau)\de\xi + \int_1^\infty\xi^\beta f_{R}(\xi,\tau)\de\xi\biggr) \de\tau \\
& \leq C\|\vphi\|_\infty \sup_{\tau\in[0,T]} \Bigl( \|f_{R}(\cdot,\tau)\|^2 + \|f_{R}(\cdot,\tau)\| M_1(f_{R}(\tau)) + \|f_{R}(\cdot,\tau)\| + M_\beta(f_{R}(\tau)) \Bigr) |t-s| \\
& \leq C(T,f_0) \|\vphi\|_\infty |t-s|,
\end{split}
\end{equation*}
where $C$ is independent of $R$ (the last estimate follows thanks to the uniform bounds \eqref{truncation10}, \eqref{truncation11}). Hence the convergence \eqref{truncation12} for every $t>0$ follows from Ascoli-Arzel\`a Theorem and a diagonal argument. Notice also that the uniform bound on the moments \eqref{truncation10} is preserved in the limit.

Finally, we need to show that we can pass to the limit in the weak formulation of the equation \eqref{weakf2} as $R_j\to\infty$ and recover \eqref{weakf}, for every test function $\vphi\in C([0,\infty))$ with $\vphi(\xi)\lesssim\xi$ as $\xi\to\infty$. Thanks to the convergence \eqref{truncation12} it is easily seen that this can be done for compactly supported test function: we obtain that $f$ satisfies \eqref{weakf2} for every $\vphi\in\cc([0,\infty))$.
Eventually, if $\vphi\in C([0,\infty))$ with $\vphi(\xi)\lesssim\xi$ as $\xi\to\infty$, we consider a cut-off $\vphi_N\in\cc([0,N+1))$ at large distance, with $\vphi_N\equiv\vphi$ in $[0,N]$, $\vphi_N\to\vphi$ as $N\to\infty$; in passing to the limit in \eqref{weakf2}, the only problematic term is the fragmentation term, which can be handled thanks to the uniform bound on the moment $M_{\theta}$:
\begin{align*}
\bigg| \int_0^T\int_0^\infty \gamma(\xi)f(\xi,t)\Bigl(\vphi(\xi)-\vphi_N(\xi)\Bigr)\de\xi\de t \bigg|
& \leq C \int_0^T\int_N^\infty \xi^{\beta+1}f(\xi,t)\de\xi\de t \\
& \leq \frac{C}{N^{\theta-(\beta+1)}} \int_0^T M_\theta(f(t))\de t \to 0 \qquad\text{as }N\to\infty.
\end{align*}
We conclude that $f$ is a weak solution in the sense of Definition~\ref{def:weakf}.
\end{proof}


\section{Logarithmic variables} \label{sect:variables}

For the analysis of the stability of solutions concentrated in peaks, it is convenient to go over to logarithmic variables, which will be used along the rest of the paper: we set
\begin{equation}\label{variables}
g(x,t) := \xi f(\xi,t), \qquad \xi=2^{x}.
\end{equation}
After an elementary change of variable, \eqref{eq:coagfrag} takes the form
\begin{equation} \label{eq:coagfrag2}
\partial_t g(x,t) = \coag[g](x,t) + \frag[g](x,t)\,,
\end{equation}
where the coagulation operator $\coag$ and the fragmentation operator $\frag$ are now given by
\begin{multline} \label{coag2}
\coag[g](x,t) := \frac{\ln2}{2}\int_{-\infty}^{x} \frac{2^xK(2^x-2^y,2^y)}{2^x-2^y} g\Bigl(\frac{\ln(2^x-2^y)}{\ln2},t\Bigr)g(y,t)\de y \\
- \ln2\int_{-\infty}^\infty K(2^x,2^y)g(x,t)g(y,t)\de y\,,
\end{multline}
\begin{multline} \label{frag2}
\frag[g](x,t) := \ln2\int_{-\infty}^\infty \frac{2^{x+y}\Gamma(2^x+2^y,2^y)}{2^x+2^y} g\Bigl(\frac{\ln(2^x+2^y)}{\ln2},t\Bigr)\de y \\
- \frac{\ln2}{2}\int_{-\infty}^x \Gamma(2^x,2^y)g(x,t)2^y\de y \,,
\end{multline}
respectively. Notice that the conservation of mass \eqref{massf} reads, in terms of $g$,
\begin{equation}\label{massg}
\int_{\R} 2^{x}g(x,t)\de x =\int_{\R}2^xg(x,0)\de x\,.
\end{equation}
We can rephrase the notion of weak solution to \eqref{eq:coagfrag} in the new variables \eqref{variables} as follows. As before, this definition guarantees that a weak solution always satisfies \eqref{massg}.

\begin{definition}[Weak solution] \label{def:weakg}
A map $g\in C([0,T];\measgbeta)$ is a \emph{weak solution} to \eqref{eq:coagfrag2} in $[0,T]$ with initial condition $g_0\in\measg$ if for every $t\in[0,T]$
\begin{equation} \label{weakg}
\begin{split}
\partial_t\biggl(\int_{\R} &g(x,t)\vphi(x)\de x \biggr) \\
& = \frac{\ln2}{2}\int_{\R}\int_{\R} K(2^y,2^z)g(y,t)g(z,t)\biggl[\vphi\Bigl(\frac{\ln(2^y+2^z)}{\ln2}\Bigr) - \vphi(y) - \vphi(z) \biggr]\de y \de z \\
& \qquad - \frac14\int_{\R} \gamma(2^{y+1}) g(y+1) \bigl[\vphi(y+1) - 2\vphi(y) \bigr]\de y
\end{split}
\end{equation}
for every test function $\vphi\in C(\R)$ such that $\lim_{x\to-\infty}\vphi(x)<\infty$ and $\vphi(x)\lesssim 2^x$ as $x\to\infty$, and $g(\cdot,0)=g_0$.
\end{definition}

It is convenient to introduce a notation for the right-hand side of the weak equation \eqref{weakg}, evaluated on a given test function $\vphi$: we define the operators
\begin{align} 
B_{\mathrm c}[g,g;\vphi]
& :=\frac{\ln2}{2}\int_{\R}\int_{\R} K(2^y,2^z)g(y)g(z)\biggl[\vphi\Bigl(\frac{\ln(2^y+2^z)}{\ln2}\Bigr) - \vphi(y) - \vphi(z) \biggr]\de y \de z \,, \label{rhsweakc} \\
B_{\mathrm f}[g;\vphi]
&:= \frac14\int_{\R} \gamma(2^{y+1}) g(y+1) \bigl[\vphi(y+1) - 2\vphi(y) \bigr]\de y \,.  \label{rhsweakf}
\end{align}

\begin{remark} \label{rm:supp}
Notice that, in view of \eqref{suppK0}, the coagulation kernel $K$ evaluated at the point $(2^y,2^z)$ is supported in the region
\begin{equation} \label{suppK}
\supp K(2^y,2^z) \subset \bigl\{ |y-z|<1 \bigr\} \,.
\end{equation}
\end{remark}

For $g\in\measg$ we can introduce the following norm, corresponding to \eqref{normf}:
\begin{equation} \label{normg}
\|g\| :=  \sup_{\substack{n\in\Z\\n< 0}} \, \frac{1}{2^n}\int_{[n,n+1)}g(x)\de x + \int_{[0,\infty)} g(x)\de x\,.
\end{equation}
For every initial datum $g_0\in\measg$ such that $\|g_0\|<\infty$ and $\int_{\R} 2^{\theta x}g(x)\de x<\infty$, for some $\theta>\beta+1$, Theorem~\ref{thm:wp} provides the existence of a global weak solution $g$ with conserved mass and satisfying for all $T>0$ the bounds
\begin{equation} \label{wpestg}
\sup_{t\in[0,T]}\|g(t)\| \leq C(T,g_0) \,, \qquad \sup_{t\in[0,T]}\int_{\R}2^{\theta x}g(x,t)\de x \leq C(T,g_0) \,.
\end{equation}


\section{Stationary solutions} \label{sect:stationary}

In this section we show the existence of a family of stationary solutions of \eqref{eq:coagfrag2} supported in a set of Dirac masses at integer distance, with a fixed mass $M>0$. The solutions to \eqref{eq:coagfrag2} we are looking for have the form
\begin{equation}\label{peak1}
g_p(x) = \sum_{n=-\infty}^\infty a_n \delta(x-n-\rho)
\end{equation}
where $\rho\in[0,1)$ is a parameter fixing the shifting of the peaks with respect to the integers, with total mass
\begin{equation} \label{peak2}
\int_{\R}2^{x}g_p(x)\de x = \sum_{n=-\infty}^\infty 2^{n+\rho}a_n = M \,.
\end{equation}

By inserting the expression \eqref{peak1} into the weak formulation \eqref{weakg} of the equation, we see that $g_p$ is a stationary solution if the coefficients $a_n$ satisfy the recurrence equation
\begin{equation} \label{eq:stat}
a_{n+1} = \zeta_{n,\rho}a_n^2\,,
\qquad\text{where }
\zeta_{n,\rho} := \frac{\ln2}{2^{n+\rho}} \frac{k(2^{n+\rho})}{\gamma(2^{n+\rho+1})} \,.
\end{equation}
To obtain \eqref{eq:stat}, we observe that in view of Remark~\ref{rm:supp} there is no interaction among different peaks in the coagulation term, hence the right-hand side of \eqref{weakg} computed on $g_p$ becomes
\begin{equation*}
\begin{split}
B_{\mathrm c}&[g_p,g_p;\vphi] - B_{\mathrm f}[g_p;\vphi] \\
& = \sum_{n=-\infty}^\infty \biggl( \frac{\ln2}{2}K(2^{n+\rho},2^{n+\rho}) a_n^2 - \frac14\gamma(2^{n+\rho+1})a_{n+1} \biggr) \Bigl( \vphi(n+\rho+1)-2\vphi(n+\rho)\Bigr) \,.
\end{split}
\end{equation*}
This quantity vanishes for every test function $\vphi$ if the coefficients $a_n$ obey \eqref{eq:stat} (recall the expression \eqref{kernel1} of the coagulation kernel $K$).

To construct a solution to \eqref{eq:stat}, it is convenient to introduce new variables
\begin{equation} \label{peak6}
\alpha_n := \frac12\zeta_{n,\rho}a_n\,,
\end{equation}
which solve the recurrence equation
\begin{equation} \label{peak7}
\alpha_{n+1} = \theta_{n,\rho}\alpha_n^2\,, \qquad \text{with }\theta_{n,\rho} := \frac{2\zeta_{n+1,\rho}}{\zeta_{n,\rho}}\,.
\end{equation}
We also observe that, in view of \eqref{peak6} and \eqref{eq:stat}, we also have
\begin{equation} \label{peak8}
\frac{a_{n+1}}{a_n}= 2\alpha_n \,.
\end{equation}

Notice that, in view of assumptions \eqref{kernel2}-\eqref{kernel2bis} and \eqref{kernel5}--\eqref{kernel5bis} on the kernels, the asymptotic behaviour of the coefficients $\zeta_{n,\rho}$ is
\begin{equation} \label{peak3}
\begin{split}
\zeta_{n,\rho}&= \Bigl(\frac{k_0\ln2}{\gamma_02^\rho}\Bigr)2^{-n} + O(2^{(\bar{c}-1)n}) \qquad\text{as }n\to-\infty, \\
\zeta_{n,\rho}&\sim (2^{-\beta}\ln2) 2^{(\alpha-\beta)(n+\rho)} \qquad\qquad\text{ as }n\to\infty,
\end{split}
\end{equation}
where $\bar{c}:=\min\{\bar{\alpha},\bar{\beta}\}>1$.
In turn the coefficients $\theta_{n,\rho}$ obey
\begin{equation} \label{peak9}
\theta_{n,\rho}= 1 + O(2^{\bar{c}n}) \quad\text{as }n\to-\infty,
\qquad\qquad
\lim_{n\to\infty}\theta_{n,\rho}=2^{\alpha-\beta+1}\,.
\end{equation}

The existence of stationary solutions in the form \eqref{peak1} with any given value of the mass is guaranteed by the following proposition.

\begin{proposition}[Stationary peaks solutions] \label{prop:stationary}
Let $\rho\in[0,1)$ and $M>0$ be given. There exists a unique family of coefficients $\{a_n\}_{n\in\Z}$ solving \eqref{eq:stat} which are positive, bounded, and satisfy
\begin{equation} \label{peak4}
\sum_{n=-\infty}^\infty 2^{n+\rho} a_n = M.
\end{equation}
Furthermore, there exist two constants $A_0\in\R$, $A_M>0$, uniquely determined by $M$, such that
\begin{equation} \label{peak5}
\begin{split}
a_n & = a_{-\infty}\bigl(2^n + A_02^{2n}\bigr) + o(2^{2n}) \qquad\qquad\text{as }n\to-\infty, \\
a_n & \sim a_\infty 2^{(\beta-\alpha)n}e^{-A_M2^n} \qquad\qquad\qquad\qquad\text{as } n\to\infty,
\end{split}
\end{equation}
where $a_{-\infty}:=\frac{\gamma_0 2^{\rho+1}}{k_0\ln2}$, $a_\infty:=(\ln2)^{-1}2^\beta2^{(\beta-\alpha)(\rho+1)}$.

In particular, the measure $g_p$ defined by \eqref{peak1} is a stationary solution to \eqref{eq:coagfrag2} with total mass $M$, that is \eqref{peak2} holds.
\end{proposition}

\begin{proof}
Given any initial value $\alpha_0>0$ we can consider the family of coefficients $\{\alpha_n\}_{n\in\Z}$ obtained by the recurrence relation \eqref{peak7} starting from $\alpha_0$. We first compute the asymptotics of $\alpha_n$ at $n\to-\infty$: by iterating \eqref{peak7} we have for all $k\in\N$
\begin{equation} \label{proofstat3}
\begin{split}
\alpha_{-k}
& = \Biggl(\prod_{\ell=1}^{k} \bigl(\theta_{\ell-k-1,\rho}\bigr)^{-2^{-\ell}}\Biggr)\alpha_0^{2^{-k}}
= \exp \Biggl( 2^{-k}\ln(\alpha_0) - \sum_{\ell=1}^k 2^{-\ell}\ln(\theta_{\ell-k-1,\rho}) \Biggr) \\
& = \exp \Biggl( 2^{-k}\ln(\alpha_0) - 2^{-k}\sum_{j=1-k}^0 2^{-j}\ln(\theta_{j-1,\rho}) \Biggr) \,.
\end{split}
\end{equation}
Observe now that, in view of the asymptotics \eqref{peak9}, the series
\begin{equation} \label{proofstat4}
\Theta := \sum_{j=-\infty}^\infty 2^{-j}\ln(\theta_{j-1,\rho})
\end{equation}
is absolutely convergent; then, by letting $k\to\infty$ in \eqref{proofstat3} we obtain the asymptotics
\begin{equation} \label{proofstat5}
\alpha_n = 1+ A_0 2^n + o(2^n) \quad\text{as }n\to-\infty,
\qquad\text{where } A_0 := \ln\alpha_0 - \sum_{j=-\infty}^{0} 2^{-j}\ln(\theta_{j-1,\rho}) \,.
\end{equation}

We now look at the behaviour of $\alpha_n$ as $n\to\infty$. For every $n\in\Z$ and $k\in\N$ we have
\begin{equation*}
\begin{split}
\alpha_n
& = \Biggl(\prod_{\ell=n-k}^{n-1}\theta_{\ell,\rho}^{2^{n-\ell-1}}\Biggr) \alpha_{n-k}^{2^k}
= \exp\Biggl( \sum_{\ell=n-k}^{n-1}2^{n-\ell-1}\ln(\theta_{\ell,\rho}) + 2^k\ln(\alpha_{n-k}) \Biggr) \\
& = \exp\Biggl( \sum_{\ell=n-k}^{n-1}2^{n-\ell-1}\ln(\theta_{\ell,\rho}) + 2^k(A_02^{n-k} + o(2^{n-k})) \Biggr) \,,
\end{split}
\end{equation*}
where we used the asymptotic \eqref{proofstat5} in the last passage; by letting $k\to\infty$ we obtain
\begin{equation*}
\begin{split}
\alpha_n
& = \exp\Biggl( \sum_{\ell=-\infty}^{n-1}2^{n-\ell-1}\ln(\theta_{\ell,\rho}) + 2^nA_0 \Biggr)
= \exp\Biggl(  2^n\sum_{j=-\infty}^{n}2^{-j}\ln(\theta_{j-1,\rho}) + 2^nA_0 \Biggr) \\
& = \exp\Biggl(  2^n \biggl(A_0+ \sum_{j=-\infty}^{\infty}2^{-j}\ln(\theta_{j-1,\rho}) \biggr)\Biggr)\exp\Biggl( -2^n\sum_{j=n+1}^{\infty}2^{-j}\ln(\theta_{j-1,\rho})\Biggr) \,,
\end{split}
\end{equation*}
which finally gives, recalling \eqref{proofstat4},
\begin{equation} \label{proofstat6}
\alpha_n = e^{-A2^n}\exp\Biggl( -2^n\sum_{j=n+1}^{\infty}2^{-j}\ln(\theta_{j-1,\rho})\Biggr)\,,
\qquad\text{where } A:= - (A_0 + \Theta) \,.
\end{equation}
Notice also that by \eqref{peak9}
\begin{equation} \label{proofstat7}
\alpha_n \sim \frac{e^{-A2^n}}{2^{\alpha-\beta+1}} \qquad\text{as }n\to\infty.
\end{equation}
Summing up, for every value of the free parameter $\alpha_0$ we have constructed a sequence of coefficients $\{\alpha_n\}_{n\in\Z}$ solving \eqref{peak7} and satisfying \eqref{proofstat5}, \eqref{proofstat6}. This family of solutions is equivalently parametrized by the exponent $A=-\ln(\alpha_0)-\sum_{j=1}^\infty 2^{-j}\ln(\theta_{j-1,\rho})$ giving the asymptotic behaviour of $\alpha_n$ at $n\to\infty$. We are interested only in those solutions with $A>0$.

We finally compute the total mass and we show that, for any given $M>0$, there is always a unique choice of the parameter $A>0$ such that the constraint \eqref{peak4} is fulfilled. In terms of $\alpha_n$, \eqref{peak4} reads
\begin{multline*}
\sum_{n=-\infty}^\infty 2^{n+\rho}a_n
= \sum_{n=-\infty}^\infty 2^{n+\rho+1}\zeta_{n,\rho}^{-1}\alpha_n \\
\xupref{proofstat6}{=} \sum_{n=-\infty}^\infty 2^{n+\rho+1}\zeta_{n,\rho}^{-1}  e^{-A2^n}\exp\Biggl( -2^n\sum_{j=n+1}^{\infty}2^{-j}\ln(\theta_{j-1,\rho})\Biggr) =: M(A) \,.
\end{multline*}
Recalling \eqref{peak3}, it is easily seen that this series is convergent for every value of $A>0$; moreover the value $M(A)$ is decreasing with respect to $A$, with $M(A)\to0$ as $A\to\infty$ and $M(A)\to\infty$ as $A\to0^+$. Therefore, given any $M>0$, we can find a unique value $A_M>0$ such that $M(A_M)=M$, so that the constraint \eqref{peak4} is satisfied.

Going back to the coefficient $a_n=2\zeta_{n,\rho}^{-1}\alpha_n$, it is easily seen that also \eqref{peak5} is satisfied, thanks to \eqref{proofstat5}, \eqref{proofstat7}, and \eqref{peak3}.
\end{proof}

\begin{remark} \label{rm:stat}
It is clear from the proof of Proposition~\ref{prop:stationary} that, for a given $\rho>0$, there is a one-to-one correspondence between the total mass $M\in(0,\infty)$ of the solution $g_p$ and the parameter $A_M\in(0,\infty)$ which determines its asymptotic behaviour at $\infty$; in other words, the result can be equivalently stated by saying that, for every $\rho>0$ and $A>0$, there exists a unique weak solution $g_p$ in the form \eqref{peak1} whose coefficients obey \eqref{peak5}. We will underline the dependence of the coefficients on the parameters $A$ and $\rho$ by writing $\{a_n(A,\rho)\}_{n\in\Z}$.
\end{remark}


\section{The linearized problem} \label{sect:linear}

The analysis in the paper and the proof of the main result rely heavily on the study of the properties of solutions to the linear problem
\begin{equation} \label{linear1}
\frac{\de y_n}{\de t} = \frac{\gamma(2^{n})}{4} \Bigl[ y_{n-1}-y_n - \sigma_n\bigl( y_n-y_{n+1}\bigr) \Bigr] ,
\end{equation}
where the coefficients $\sigma_n$ are given by
\begin{equation} \label{linear2}
\sigma_n := 8\alpha_n(A_M)\frac{\gamma(2^{n+1})}{\gamma(2^n)}
\end{equation}
(recall here that $\alpha_n$ are the rescaled coefficients of a stationary solution corresponding to the parameters $A_M$ and $\rho=0$, introduced in \eqref{peak6} and explicitly given by \eqref{proofstat6}). The equation \eqref{linear1} will appear in Section~\ref{sect:stability1} as the linearization of the evolution equation for the coefficients of a weak solution concentrated in Dirac masses. The value of $M>0$, and consequently of $A_M$, is fixed throughout this section. Notice that, in view of \eqref{kernel5}, \eqref{kernel5bis}, \eqref{proofstat5}, and \eqref{proofstat6}, we have the asymptotics
\begin{equation} \label{linear3}
\sigma_n\sim 8 -8(A_M+\Theta)2^n+o(2^n) \quad\text{as $n\to-\infty$,}
\qquad
\sigma_n\sim 2^{2\beta-\alpha+2}e^{-A_M 2^n} \quad\text{as $n\to\infty$,}
\end{equation}
where $\Theta$ is the constant appearing in \eqref{proofstat4}.

We fix some notation to be used in the following. We introduce the space of sequences which are bounded as $n\to\infty$ and grow at most as $2^{-n}$ as $n\to-\infty$:
\begin{equation} \label{linearspace}
\mathcal{Y} := \Bigl\{ y=\{y_n\}_{n\in\Z} \,:\, \|y\|_{\mathcal{Y}}<\infty \Bigr\},
\qquad
\|y\|_{\mathcal{Y}} := \sup_{n\leq0} \, 2^{n}|y_n| + \sup_{n>0} \, |y_n| \,.
\end{equation}
For $\theta\in\R$ we also consider the space
\begin{equation} \label{linearspace2}
\mathcal{Y}_\theta := \Bigl\{ y=\{y_n\}_{n\in\Z}\in\mathcal{Y} \,:\, \|y\|_{\theta}<\infty \Bigr\},
\qquad
\|y\|_{\theta} := \sup_{n\leq0} \, 2^{n}|y_n| + \sup_{n>0} \, 2^{\theta n}|y_n| 
\end{equation}
(notice in particular that $\mathcal{Y}_0=\mathcal{Y}$).
It is convenient to denote the linear operator on the right-hand side of \eqref{linear1}, acting on a sequence $y=\{y_n\}_{n\in\Z}$, by
\begin{equation} \label{linear4}
\mathscr{L}_n(y) := \frac{\gamma(2^{n})}{4} \Bigl[ y_{n-1}-y_n - \sigma_n\bigl( y_n-y_{n+1}\bigr) \Bigr], 
\qquad
\mathscr{L}(y) := \{ \mathscr{L}_n(y) \}_{n\in\Z} \,,
\end{equation}
and the associated semigroup by
\begin{equation} \label{linear10}
S_n(t)(y^0) := y_n(t),
\qquad
S(t)(y^0) := \{S_n(t)(y^0)\}_{n\in\Z} \,,
\end{equation}
where $\{y_n(t)\}_{n\in\Z}$ is the unique solution to \eqref{linear1} at time $t$, with initial datum $y^0$.
We finally introduce a symbol for the discrete derivatives
\begin{equation}  \label{linearD}
D^+_n(y) := y_{n+1}-y_n,
\qquad
D^-_n(y) := y_n-y_{n-1},
\qquad
D^{\pm}(y) := \{ D^{\pm}_n(y) \}_{n\in\Z} \,,
\end{equation}
and for the operator
\begin{equation} \label{linear7}
P_n(y) := 2^ny_n, \qquad P(y):=\{ P_n(y) \}_{n\in\Z} \,.
\end{equation}

We summarize the main regularity result for the linear equation \eqref{linear1} in the following statement, whose technical proof is postponed to Appendix~\ref{sect:appendix}.

\begin{theorem} \label{thm:linear}
Let $\theta$, $\tilde{\theta}$ be fixed parameters satisfying
\begin{equation} \label{theta}
\theta\in(-1,\beta), \qquad \tilde{\theta}\in[\theta,\beta], \qquad\text{with }\tilde{\theta}-\theta<\beta,
\end{equation}
and let $y^0\in\mathcal{Y}_\theta$ be a given initial datum.
Then there exists a unique solution $t\mapsto S(t)(y^0)\in\mathcal{Y}_\theta$ to the linear problem \eqref{linear1} with initial datum $y^0$, which satisfies the estimate
\begin{equation} \label{linear13b}
\| D^+(S(t)(y^0)) \|_{\tilde{\theta}} \leq C_1 \|y^0\|_\theta  (1+t^{-\frac{\tilde{\theta}-\theta}{\beta}}) e^{-\nu t} \qquad\text{for all $t>0$,}
\end{equation}
where $\nu>0$ is a constant depending only on $M$, and $C_1>0$ depends on $M$, $\theta$, and $\tilde{\theta}$.

Moreover, there exists the limit $S_{\infty}(t)(y^0):=\lim_{n\to\infty}S_n(t)(y^0)$, satisfying for all $t>0$
\begin{equation} \label{linear13}
\|S(t)(y^0)-S_{\infty}(t)(y^0)\|_{\tilde{\theta}} \leq C_1 \|y^0\|_\theta  (1+t^{-\frac{\tilde{\theta}-\theta}{\beta}}) e^{-\nu t} \qquad\text{(if $\tilde{\theta}>0$).}
\end{equation}

Finally, there exists the limit
\begin{equation*}
D^\beta_\infty S(t)(y^0) := \lim_{n\to\infty} 2^{\beta n}\bigl( S_n(t)(y^0) - S_\infty(t)(y^0) \bigr) \,,
\end{equation*}
and one has the identity
\begin{equation} \label{linear12}
\lim_{n\to\infty} \mathscr{L}_n(S(t)(y^0)) = \frac{2^\beta-1}{4}D^\beta_\infty S(t)(y^0).
\end{equation}
\end{theorem}

All the constants in the statement above depend also on the properties of the coagulation and fragmentation kernels; however we will not mention this dependence explicitly, as the kernels are fixed throughout the paper. For the purposes of this paper, we will not need to consider negative values of $\theta$ in the previous theorem (that is, initial data which are unbounded as $n\to\infty$); however, since we will need this result for $\theta<0$ in the companion paper \cite{BNVe}, we prefer to include also this possibility in the statement. 


\section{Stability in the class of peaks solutions} \label{sect:stability1}

\subsection{The main result} \label{subsect:stability1main}
In this section we discuss the stability of the stationary solutions with peaks constructed in Proposition~\ref{prop:stationary}, among the class of weak solutions concentrated in Dirac masses. We consider a weak solution $g\in C([0,T];\measgbeta)$ to \eqref{eq:coagfrag2}, according to Definition~\ref{def:weakg}, in the form
\begin{equation} \label{dirac1}
g(x,t) = \sum_{n=-\infty}^\infty b_n(t)\delta(x-n-\rho) \,,
\end{equation}
where $\rho\in[0,1)$ is fixed. We denote the total mass of the solution, which is preserved along the evolution, by
\begin{equation} \label{dirac1bis}
M:=\int_{\R}2^xg(x,t)\de x = \sum_{n=-\infty}^\infty 2^{n+\rho}b_n(t) \,.
\end{equation}
In order to derive the evolution equation for the coefficients $b_n(t)$, we choose a sequence of test functions $\vphi_n(x):=2^x\psi(x-n-\rho)$, where $\psi\in C^\infty_{\mathrm{c}}(-\frac12,\frac12)$, $\psi(0)=1$, is a cut-off function; then by inserting the expression \eqref{dirac1} into the weak formulation \eqref{weakg} of the equation we obtain
\begin{equation*}
\begin{split}
2^{n+\rho}\frac{\de b_n}{\de t}
& = \sum_{k=-\infty}^\infty \frac{\ln2}{2}K(2^{k+\rho},2^{k+\rho}) (b_k(t))^2 \bigl( \vphi_n(k+\rho+1)-2\vphi_n(k+\rho)\bigr) \\
& \qquad -  \sum_{k=-\infty}^\infty \frac{\gamma(2^{k+\rho+1})}{4}b_{k+1}(t)\bigl( \vphi_n(k+\rho+1)-2\vphi_n(k+\rho)\bigr) \\
& =  \biggl( \frac{\ln2}{2}K(2^{n-1+\rho},2^{n-1+\rho}) (b_{n-1}(t))^2 - \frac14\gamma(2^{n+\rho})b_{n}(t) \biggr)2^{n+\rho} \\
& \qquad - \biggl( \frac{\ln2}{2}K(2^{n+\rho},2^{n+\rho}) (b_{n}(t))^2 - \frac14\gamma(2^{n+1+\rho})b_{n+1}(t) \biggr) 2^{n+\rho+1} \\
& = \biggl( \ln2\frac{k(2^{n-1+\rho})}{2^{n+1+\rho}} (b_{n-1}(t))^2 - \frac14\gamma(2^{n+\rho})b_{n}(t) \biggr) 2^{n+\rho}\\
& \qquad - \biggl( \ln2\frac{k(2^{n+\rho})}{2^{n+2+\rho}} (b_{n}(t))^2 - \frac14\gamma(2^{n+1+\rho})b_{n+1}(t) \biggr) 2^{n+\rho+1} \,,
\end{split}
\end{equation*}
where in the first passage we used the fact that the function $K(2^y,2^z)$ is supported in the region $|y-z|<1$, see Remark~\ref{rm:supp}. By using the coefficients $\zeta_{n,\rho}$ introduced in \eqref{eq:stat}, we can write the previous equation as
\begin{equation} \label{dirac2}
\frac{\de b_n}{\de t} = \frac{\gamma(2^{n+\rho})}{4}\biggl( \zeta_{n-1,\rho}(b_{n-1}(t))^2 - b_n(t) \biggr) - \frac{\gamma(2^{n+1+\rho})}{2}\biggl( \zeta_{n,\rho}(b_{n}(t))^2 - b_{n+1}(t) \biggr) \,.
\end{equation}

Recall now that by Proposition~\ref{prop:stationary} we have a two-parameter family of stationary solutions in the form $g_p(x;A,\rho)=\sum_{n=-\infty}^\infty a_n(A,\rho)\delta(x-n-\rho)$, parametrized by $A>0$, $\rho\in[0,1)$ (see Remark~\ref{rm:stat}). The main goal is to express the coefficients $\{b_n(t)\}_{n\in\Z}$ of the weak solution \eqref{dirac1} as perturbations of the coefficients $\{a_n(A(t),\rho)\}_{n\in\Z}$ of one of these stationary solutions, at each time corresponding to a different value of the parameter $A$. As $t\to\infty$, the weak solution $g$ approaches the stationary solution $g_p(\cdot;A_M,\rho)$, where $A_M$ is the unique value of the parameter $A$ such that the corresponding solution has total mass $M$. Precisely, our main result is the following.

\begin{theorem} \label{thm:dirac}
Let $\{b_n^0\}_{n\in\Z}$ be an initial datum in the form
\begin{equation} \label{Dirac3}
b_n^0=a_n(A^0,\rho)\bigl(1+\e_n^0\bigr)
\end{equation}
for some $A^0>0$ and $\e_n^0\in\R$, and let $M=\sum_{n=-\infty}^\infty2^{n+\rho}b_n^0$ be the initial mass.
There exists $\delta_0>0$, depending on $M$, such that if 
\begin{equation} \label{Dirac4}
|A^0-A_M|\leq\delta_0,\qquad |\e_n^0|\leq\delta_0,
\end{equation}
then there exists a solution $\{b_n(t)\}_{n\in\Z}$ to \eqref{dirac2}, with $b_n(0)=b_n^0$, which can be expressed as
\begin{equation} \label{Dirac}
b_n(t) = a_n(A(t),\rho) \bigl( 1+\e_n(t) \bigr) \qquad\text{for every $t>0$ and $n\in\Z$,}
\end{equation}
for suitable functions $A(t)>0$ and $\e_n(t)\in\R$. Moreover for every $t>0$ one has the estimates
\begin{equation}\label{Dirac2}
\sup_{n\leq0}|\e_n(t)| + \sup_{n\geq0}2^{(\beta-1)n}|\e_n(t)| \leq (1+t^{-\frac{\beta-1}{\beta}})e^{-\nu t},
\qquad
\Big|\frac{\de A}{\de t}\Big| \leq (1+t^{-\frac{\beta-1}{\beta}})e^{-\nu t},
\end{equation}
where $\nu>0$ depends only on $M$. Finally, $A(t)\to A_M$ as $t\to\infty$.
\end{theorem}

In the statement of the theorem, $\{a_n(A(t),\rho)\}_{n\in Z}$ denote the coefficients of the stationary solution constructed in Proposition~\ref{prop:stationary} corresponding to the parameters $A(t)$ and $\rho$, and $A_M$ is the unique value of $A$ such that $\sum_{n}2^{n+\rho}a_n(A_M,\rho)=M$.

\medskip

The rest of this section is devoted to the proof of Theorem~\ref{thm:dirac}. We will first write the equation in a more convenient set of variables and outline the strategy of the proof, and then perform the proof in full details by means of a fixed point argument.

In order to simplify the notation, we will assume henceforth that $\rho=0$, without loss of generality, and we will not indicate the dependence on $\rho$.
Recall that the goal is to express the coefficients $\{b_n(t)\}_{n\in\Z}$, solving \eqref{dirac2}, in the form
\begin{equation} \label{dirac3}
b_n(t) = a_n(A(t)) \bigl( 1+\e_n(t) \bigr)
\end{equation}
for suitable functions $A(t)$ and $\e_n(t)$, whose existence will be proved by means of a fixed point argument. We substitute \eqref{dirac3} into \eqref{dirac2}: by using the fact that the sequence $(a_n)_{n\in\N}$ satisfies the equation \eqref{eq:stat} we easily get
\begin{multline} \label{dirac4}
\frac{\de a_n}{\de A}\frac{\de A}{\de t} \bigl(1+\e_n\bigr) + a_n(A(t))\frac{\de\e_n}{\de t} \\
= \frac{\gamma(2^{n})}{4} \biggl[ a_n(A(t)) \Bigl( 2\e_{n-1} -\e_n + \e_{n-1}^2 \Bigr) - 2a_{n+1}(A(t))\frac{\gamma(2^{n+1})}{\gamma(2^{n})} \Bigl( 2\e_n-\e_{n+1}+\e_n^2 \Bigr) \biggr] \,.
\end{multline} 
Notice now that, in view of \eqref{peak6} and \eqref{proofstat6}, we have
\begin{equation*}
\frac{\de a_n}{\de A}=\frac{2}{\zeta_n}\frac{\de \alpha_n}{\de A} =-\frac{2}{\zeta_n} 2^n \alpha_n = -2^na_n\,,
\end{equation*}
and furthermore $\frac{a_{n+1}}{a_n}=2\alpha_n$ by \eqref{peak8}. Inserting these identities into \eqref{dirac4} we find
\begin{multline} \label{dirac5}
\frac{\de\e_n}{\de t}
= 2^n(1+\e_n)\frac{\de A}{\de t} \\
+ \frac{\gamma(2^{n})}{4} \biggl[ \Bigl( 2\e_{n-1} -\e_n + \e_{n-1}^2 \Bigr) - 4\alpha_n(A(t))\frac{\gamma(2^{n+1})}{\gamma(2^{n})}\Bigl( 2\e_n-\e_{n+1}+\e_n^2 \Bigr) \biggr] \,.
\end{multline}
It is now convenient to introduce the variable $y_n:=2^{-n}\e_n$, solving
\begin{multline} \label{dirac8}
\frac{\de y_n}{\de t}
= (1+2^ny_n)\frac{\de A}{\de t} \\
+ \frac{\gamma(2^{n})}{4} \biggl[ \Bigl( y_{n-1} -y_n + 2^{n-2}y_{n-1}^2 \Bigr) - 4\alpha_n(A(t))\frac{\gamma(2^{n+1})}{\gamma(2^{n})}\Bigl( 2y_n-2y_{n+1}+2^ny_n^2 \Bigr) \biggr] \,.
\end{multline}
Notice also that the mass-conservation property \eqref{dirac1bis} gives the additional constraint
\begin{equation} \label{dirac6}
M = \sum_{n=-\infty}^\infty 2^na_n(A(t)) \bigl( 1+ 2^n y_n(t)\bigr) \qquad\text{for every $t>0$.}
\end{equation}


\subsection{Strategy of the proof} \label{subsect:stability1strategy}
The proof of Theorem~\ref{thm:dirac} is based on the analysis of the linearized version of the equation \eqref{dirac8}, which is considered in Section~\ref{sect:linear}. We refer to the beginning of that section for the notation that will be used in the following. The equation \eqref{dirac8} for the coefficients $y(t)=\{y_n(t)\}_{n\in\Z}$ and the unknown function $A(t)$ can be written in terms of the linearized operator $\mathscr{L}$ (see \eqref{linear4}) as
\begin{equation} \label{fp1}
\frac{\de y_n}{\de t}(t)
= \mathscr{L}_n(y(t)) + \frac{\de A}{\de t}(t) + 2^ny_n(t)\frac{\de A}{\de t}(t) + h_n(y(t),A(t)) ,
\end{equation}
where for notational convenience we introduced the sequence $h(y,A):=\{ h_n(y,A) \}_{n\in\Z}$,
\begin{multline} \label{fp2}
h_n(y,A) := \frac{2^n\gamma(2^{n})}{4} \biggl[ \frac14 y_{n-1}^2 - 4\alpha_n(A)\frac{\gamma(2^{n+1})}{\gamma(2^{n})} y_n^2 \biggr] \\
+ 2\gamma(2^{n+1}) \bigl(\alpha_n(A_M)-\alpha_n(A)\bigr)\bigl( y_n-y_{n+1} \bigr) .
\end{multline}

Before going to the technical details of the proof, we briefly explain the general strategy: it consists in selecting the function $A(t)$ in such a way that the solution to \eqref{fp1} satisfies $\lim_{n\to\infty}y_n(t)=0$ for every positive time. The existence of a pair of functions $(y(t),A(t))$ with this property will be proved by means of a fixed point argument.
In order to formulate the problem as a fixed point, we notice that, for a given $A(t)$, we can write a representation formula for the solution $y(t)=\{y_n(t)\}_{n\in\N}$ to \eqref{fp1} with a given initial datum $y^0$:
\begin{multline} \label{fp3}
y_n(t) = S_n(t)(y^0) + A(t) - A(0) + \int_0^t \frac{\de A}{\de t}(s)S_n(t-s)(P(y(s)))\de s \\
+ \int_0^t S_n(t-s)(h(y(s),A(s)))\de s ,
\end{multline}
where $S_n$ is the semigroup generated by the linearized equation \eqref{linear1}, see \eqref{linear10}, and $P$ is the operator introduced in \eqref{linear7}. By formally taking the derivative with respect to $t$ in \eqref{fp3} we get
\begin{multline} \label{fp4}
\frac{\de y_n}{\de t}(t)
= \mathscr{L}_n\bigl(S(t)(y^0)\bigr) + \frac{\de A}{\de t} + 2^ny_n(t)\frac{\de A}{\de t} + h_n(y(t),A(t)) \\
+ \int_0^t \frac{\de A}{\de t}(s) \mathscr{L}_n\bigl( S(t-s)( P(y(s)) ) \bigr) \de s 
+ \int_0^t \mathscr{L}_n\bigl( S(t-s)(h(y(s),A(s))) \bigr) \de s \,.
\end{multline}
We obtain an equation defining $A(t)$ by imposing that the formal limit as $n\to\infty$ of \eqref{fp4} vanishes: recalling \eqref{linear12}, we have
\begin{multline} \label{fp5}
\frac{\de A}{\de t}
= \frac{1-2^\beta}{4}\biggl[ D^\beta_\infty S(t)(y^0)
+ \int_0^t \frac{\de A}{\de t}(s) D^\beta_\infty S(t-s)( P(y(s)) ) \de s \\
+ \int_0^t D^\beta_\infty S(t-s)(h(y(s),A(s))) \de s \biggr]
\end{multline}
(where we assumed that the limits of $2^ny_n(t)$ and of $h_n(y(t),A(t))$ vanish: this will be a consequence of the choice of the space in which we formulate the fixed point). We will show that all the quantities appearing in \eqref{fp5} are well-defined, and moreover that \eqref{fp5} implies that the right-hand side of \eqref{fp3} vanishes in the limit as $n\to\infty$, that is,
\begin{multline} \label{fp6}
S_\infty(t)(y^0) + A(t) - A(0) + \int_0^t \frac{\de A}{\de t}(s)S_\infty(t-s)( P(y(s)) )\de s \\
+ \int_0^t S_\infty(t-s)(h(y(s),A(s)))\de s = 0 \,.
\end{multline}
Finally, subtracting \eqref{fp6} from \eqref{fp3}, we obtain the equation
\begin{multline}  \label{fp7}
y_n(t)
= \bigl[S_n(t)-S_\infty(t)\bigr](y^0) + \int_0^t \frac{\de A}{\de t}(s) \bigl[S_n(t-s)-S_\infty(t-s)\bigr]( P(y(s)) )\de s \\
+ \int_0^t \bigl[S_n(t-s)-S_\infty(t-s)\bigr](h(y(s),A(s)))\de s \,.
\end{multline}
The previous formal considerations suggests to look for a pair $(y(t), \Lambda(t))$, with $\Lambda(t)=\frac{\de A}{\de t}$, solving the two equations \eqref{fp5}--\eqref{fp7} by a fixed point argument. This is made rigorous in the following subsection.


\subsection{Fixed point argument} \label{sect:fixedpoint}
The core of the proof of Theorem~\ref{thm:dirac} is contained in the following proposition.

\begin{proposition} \label{prop:fixedpoint}
Let $M>0$. There exists $\delta_0>0$, depending on $M$, with the following property. If $A^0>0$, $y^0\in\mathcal{Y}_1$ are such that
\begin{equation} \label{fp8}
M=\sum_{n=-\infty}^\infty 2^na_n(A^0)(1+2^ny_n^0)
\end{equation}
and
\begin{equation} \label{fp9}
|A^0-A_M|\leq\delta_0,\qquad \|y^0\|_1\leq\delta_0,
\end{equation}
then there exist functions $t\mapsto (y(t),A(t))$ solving \eqref{fp1} with $y(0)=y^0$, $A(0)=A^0$. Moreover for every $t>0$ one has the estimates
\begin{equation}\label{fp10}
\|y(t)\|_\beta \leq (1+t^{-\frac{\beta-1}{\beta}})e^{-\frac{\nu}{2} t},
\qquad
\Big|\frac{\de A}{\de t}\Big| \leq (1+t^{-\frac{\beta-1}{\beta}})e^{-\frac{\nu}{2}t},
\end{equation}
where $\nu$ is as in Theorem~\ref{thm:linear} (depending only on $M$). Finally, $A(t)\to A_M$ as $t\to\infty$.
\end{proposition}

\begin{proof}
Let $g:(0,\infty)\to\R$ be the function
\begin{equation} \label{pfp0}
g(t) := (1+t^{-\frac{\beta-1}{\beta}})e^{-\frac{\nu}{2}t}.
\end{equation}
Let also $\delta>0$ be a small parameter, to be chosen later.
We work in the space $\mathcal{X}:=\mathcal{X}_1\times\mathcal{X}_2$, where
\begin{equation} \label{pfp1}
\mathcal{X}_1 := \Bigl\{ y\in L^\infty_{\loc}((0,\infty);\mathcal{Y}_\beta) \,:\, \|y\|_{\mathcal{X}_1}\leq \delta \Bigr\},
\qquad
\|y\|_{\mathcal{X}_1} := \sup_{t>0} \, \frac{\|y(t)\|_\beta}{g(t)} \,,
\end{equation}
\begin{equation} \label{pfp2}
\mathcal{X}_2 := \Bigl\{ \Lambda\in L^\infty_{\loc}((0,\infty);\R) \,:\, \|\Lambda\|_{\mathcal{X}_2}\leq \delta \Bigr\},
\qquad
\|\Lambda\|_{\mathcal{X}_2} := \sup_{t>0} \, \frac{|\Lambda(t)|}{g(t)} \,.
\end{equation}
For $\Lambda\in\mathcal{X}_2$ we let
\begin{equation} \label{pfp5}
A_\Lambda(t) := A^0 + \int_0^t \Lambda(s)\de s \,.
\end{equation}
Notice that for every $\Lambda\in\mathcal{X}_2$
\begin{equation} \label{pfp5bis}
|A_\Lambda(t)-A^0| = \bigg|\int_0^t\Lambda(s)\de s\bigg| \leq \|\Lambda\|_{\mathcal{X}_2}\int_0^t g(s)\de s \leq C_3 \|\Lambda\|_{\mathcal{X}_2} \,,
\end{equation}
where $C_3:=\int_0^\infty g(t)\de s<\infty$ depends ultimately only on $M$. In particular, combining \eqref{fp9} and \eqref{pfp5bis} we can assume without loss of generality that for every $\Lambda\in\mathcal{X}_2$
\begin{equation}  \label{pfp5ter}
|A_{\Lambda}(t)-A_M| \leq C_3\delta + \delta_0 \leq \frac{A_M}{2}
\end{equation}
provided that we choose $\delta_0$ and $\delta$ sufficiently small (depending on $M$).

We define a map $\mathcal{T}:\mathcal{X}\to\mathcal{X}$ by setting $\mathcal{T}(y,\Lambda):=(\tilde{y},\tilde{\Lambda})$, where
\begin{multline} \label{pfp3}
\tilde{y}_n(t) := \bigl[S_n(t)-S_\infty(t)\bigr](y^0) + \int_0^t \Lambda(s) \bigl[S_n(t-s)-S_\infty(t-s)\bigr]( P(y(s)) )\de s \\
+ \int_0^t \bigl[S_n(t-s)-S_\infty(t-s)\bigr](h(y(s),A_\Lambda(s)))\de s \,,
\end{multline}
\begin{multline} \label{pfp4}
\tilde{\Lambda}(t) :=
\frac{1-2^\beta}{4}\biggl[ D^\beta_\infty S(t)(y^0)
+ \int_0^t \Lambda(s) D^\beta_\infty S(t-s)( P(y(s)) ) \de s \\
+ \int_0^t D^\beta_\infty S(t-s)(h(y(s),A_\Lambda(s))) \de s \biggr] \,.
 \end{multline}

We first have to show that all the quantities appearing on the right-hand side of \eqref{pfp3} and \eqref{pfp4} are well-defined.
Since $y^0\in\mathcal{Y}_1$, by Theorem~\ref{thm:linear} the limits $S_\infty(t)(y^0)$, $D_\infty^\beta S(t)(y^0)$ exist, with the bounds
\begin{equation} \label{pfp10}
\begin{split}
\|S_n(t)(y^0)-S_\infty(t)(y^0)\|_\beta & \leq C_1\|y^0\|_1\bigl(1+t^{-\frac{\beta-1}{\beta}}\bigr)e^{-\nu t} \leq C_1\delta_0 g(t), \\
|D_\infty^\beta S(t)(y^0)| & \leq C_1\delta_0 g(t).
\end{split}
\end{equation}
Similarly, since $y(s)\in\mathcal{Y}_\beta$ we have $P(y(s))\in\mathcal{Y}_{\beta-1}$ for every positive $s$, with
\begin{equation} \label{pfp13}
\|P(y(s))\|_{\beta-1} \leq \|y(s)\|_\beta,
\end{equation}
and therefore again by Theorem~\ref{thm:linear}
\begin{equation} \label{pfp11}
\begin{split}
\| \bigl(S_n(t-s)-S_\infty(t-s)\bigr)(P(y(s)))\|_\beta &\leq C_1\|y(s)\|_\beta\bigl(1+(t-s)^{-\frac{1}{\beta}}\bigr)e^{-\nu(t-s)} \\
& \leq C_1\|y\|_{\mathcal{X}_1}g(s)\bigl(1+(t-s)^{-\frac{1}{\beta}}\bigr)e^{-\nu(t-s)} \\
& \leq C_1\delta g(s)\bigl(1+(t-s)^{-\frac{1}{\beta}}\bigr)e^{-\nu(t-s)} , \\
|D_\infty^\beta S(t-s)(P(y(s)))| &\leq C_1\delta g(s)\bigl(1+(t-s)^{-\frac{1}{\beta}}\bigr)e^{-\nu(t-s)} .
\end{split}
\end{equation}
In the same way, the bound on $h$ proved in Lemma~\ref{lem:fixedpointh} below (notice that the assumption of the lemma is satisfied in view of \eqref{pfp5ter}) also guarantees that
\begin{equation} \label{pfp12}
\begin{split}
\|\bigl(S_n(t-s)-&S_\infty(t-s)\bigr)(h(y(s),A_\Lambda(s)))\|_\beta \\
& \leq C_1\|h(y(s),A_\Lambda(s))\|_{\beta-1} \bigl(1+(t-s)^{-\frac{1}{\beta}}\bigr)e^{-\nu(t-s)} \\
& \leq C_1C_2 \Bigl( \|y(s)\|_{\beta}^2 + |A_{\Lambda}(s)-A_M|\|y(s)\|_\beta \Bigr)\bigl(1+(t-s)^{-\frac{1}{\beta}}\bigr)e^{-\nu(t-s)} \\
& \leq C_1C_2 \Bigl( \|y\|_{\mathcal{X}_1}^2\bigl(g(s)\bigr)^2 + \bigl(C_3\delta + \delta_0\bigr)\|y\|_{\mathcal{X}_1}g(s) \Bigl)\bigl(1+(t-s)^{-\frac{1}{\beta}}\bigr)e^{-\nu(t-s)} \\
& \leq C_1C_2 \Bigl( \delta^2(g(s))^2 + \bigl( C_3\delta + \delta_0 \bigr)\delta g(s) \Bigr) \bigl(1+(t-s)^{-\frac{1}{\beta}}\bigr)e^{-\nu(t-s)} , \\
| D^\beta_\infty S(t-s)(&h(y(s),A_\Lambda(s))) | \\
& \leq C_1C_2 \Bigl( \delta^2(g(s))^2 + \bigl( C_3\delta + \delta_0 \bigr)\delta g(s) \Bigr) \bigl(1+(t-s)^{-\frac{1}{\beta}}\bigr)e^{-\nu(t-s)} .
\end{split}
\end{equation}
The previous estimates imply that the right-hand sides of \eqref{pfp3} and \eqref{pfp4} are well-defined quantities. The rest of the proof consists in showing that $\mathcal{T}$ is a contraction in the space $\mathcal{X}$, and that its fixed point is the sought solution to \eqref{fp1} satisfying $\lim_{n\to\infty}y_n(t)=0$.

\smallskip\noindent\textit{Step 1: $\mathcal{T}(y,\Lambda)\in\mathcal{X}$ for every $(y,\Lambda)\in\mathcal{X}$.}
By plugging the estimates \eqref{pfp10}, \eqref{pfp11}, and \eqref{pfp12} into \eqref{pfp3} we find
\begin{equation} \label{pfp16}
\begin{split}
\|\tilde{y}_n(t)\|_\beta
& \leq C_1\delta_0 g(t)
+ C_1\bigl(1+C_2\bigr)\delta^2 \int_0^t \bigl(g(s)\bigr)^2\bigl(1+(t-s)^{-\frac{1}{\beta}}\bigr)e^{-\nu(t-s)} \de s \\
& \qquad + C_1C_2\bigl( C_3\delta + \delta_0 \bigr)\delta \int_0^t g(s)\bigl(1+(t-s)^{-\frac{1}{\beta}}\bigr)e^{-\nu(t-s)}\de s \\
& \leq C_1\Bigl( \delta_0 + (1+C_2)C_4\delta^2 + C_2C_3C_4\delta^2 + C_2C_4\delta_0\delta \Bigr)g(t) \,,
\end{split}
\end{equation}
where we used the elementary estimates
\begin{equation} \label{pfp15}
\begin{split}
\int_0^t \bigl(g(s)\bigr)^2\bigl(1+(t-s)^{-\frac{1}{\beta}}\bigr)e^{-\nu(t-s)} \de s &\leq C_4 g(t),\\
\int_0^t g(s)\bigl(1+(t-s)^{-\frac{1}{\beta}}\bigr)e^{-\nu(t-s)}\de s &\leq C_4 g(t)
\end{split}
\end{equation}
for a uniform constant $C_4$. In particular, by choosing $\delta$ and $\delta_0$ sufficiently small, depending only on $M$, \eqref{pfp16} yields
\begin{equation*}
\|\tilde{y}\|_{\mathcal{X}_1} \leq \delta.
\end{equation*}

We similarly estimate $\tilde{\Lambda}$ in the space $\mathcal{X}_2$: by inserting \eqref{pfp10}, \eqref{pfp11}, and \eqref{pfp12} into \eqref{pfp4}, the very same computation as in \eqref{pfp16} gives
\begin{equation*}
\begin{split}
|\tilde{\Lambda}(t)|
& \leq \frac{2^\beta-1}{4} C_1\Bigl( \delta_0 + (1+C_2)C_4\delta^2 + C_2C_3C_4\delta^2 + C_2C_4\delta_0\delta \Bigr)g(t) \,,
\end{split}
\end{equation*}
and in turn
\begin{equation*}
\|\tilde{\Lambda}\|_{\mathcal{X}_2} \leq \delta.
\end{equation*}
This completes the proof of the claim.

\smallskip\noindent\textit{Step 2: contractivity.}
Let $(y^1,\Lambda^1),(y^2,\Lambda^2)\in\mathcal{X}$ and set $(\tilde{y}^i,\tilde{\Lambda}^i):=\mathcal{T}(y^i,\Lambda^i)$, $i=1,2$. In view of the definition \eqref{pfp3} of $\tilde{y}^i$ we have
\begin{equation} \label{pfp14}
\begin{split}
\big| \tilde{y}_n^1(t)-\tilde{y}_n^2(t) \big|
& \leq \int_0^t \big| \Lambda^1(s)-\Lambda^2(s)\big| \big|S_n(t-s)-S_\infty(t-s)\big|( P(y^1(s)) )\de s \\
& +  \int_0^t |\Lambda^2(s)| \big|S_n(t-s)-S_\infty(t-s)\big| \bigl( P(y^1(s)-y^2(s)) \bigr)\de s \\
& + \int_0^t \big|S_n(t-s)-S_\infty(t-s)\big| \bigl( h(y^1(s),A_{\Lambda^1}(s))-h(y^2(s),A_{\Lambda^2}(s)) \bigr)\de s \,.
\end{split}
\end{equation}
The first two integrals can be estimated using \eqref{pfp11}; for the last one, by Lemma~\ref{lem:fixedpointh} below we have, similarly to \eqref{pfp12},
\begin{equation*}
\begin{split}
\big\| \bigl(S_n(t-s)-&S_\infty(t-s)\bigr) \bigl( h(y^1(s),A_{\Lambda^1}(s))-h(y^2(s),A_{\Lambda^2}(s)) \bigr) \big\|_\beta \\
& \leq C_1\| h(y^1(s),A_{\Lambda^1}(s))-h(y^2(s),A_{\Lambda^2}(s)) \|_{\beta-1} \bigl(1+(t-s)^{-\frac{1}{\beta}}\bigr)e^{-\nu(t-s)} \\
& \leq C_1C_2 \biggl[ \|y^1(s)-y^2(s)\|_\beta \Bigl( \max\{ \|y^1(s)\|_\beta, \|y^2(s)\|_\beta\} + |A_M-A_{\Lambda^1}(s)| \Bigr) \\
& \qquad\qquad + |A_{\Lambda^1}(s)-A_{\Lambda^2}(s)| \Bigl( \|y^2(s)\|_\beta^2 + \|y^2(s)\|_\beta \Bigr) \biggr] \bigl(1+(t-s)^{-\frac{1}{\beta}}\bigr)e^{-\nu(t-s)} \\
& \leq C_1C_2 \biggl[ \|y^1-y^2\|_{\mathcal{X}_1}g(s) \Bigl( \delta g(s) + C_3\delta + \delta_0 \Bigr) \\
& \qquad\qquad + C_3\|\Lambda^1-\Lambda^2\|_{\mathcal{X}_2} \Bigl( \delta^2\bigl(g(s)\bigr)^2 + \delta g(s)\Bigr) \biggr]\bigl(1+(t-s)^{-\frac{1}{\beta}}\bigr)e^{-\nu(t-s)} \,,
\end{split}
\end{equation*}
where we used the bound $|A_{\Lambda^1}(s)-A_{\Lambda^2}(s)|\leq C_3\|\Lambda^1-\Lambda^2\|_{\mathcal{X}_2}$, which follows from \eqref{pfp5}.
From \eqref{pfp14} it is then straightforward to obtain an estimate of the form
\begin{multline*}
\big\| \tilde{y}^1(t) - \tilde{y}^2(t) \big\|_{\beta}
\leq C(\delta+\delta_0) \Bigl( \|\Lambda^1-\Lambda^2\|_{\mathcal{X}_2} + \|y^1-y^2\|_{\mathcal{X}_1} \Bigr)\\
\int_0^t \bigl[ \bigl(g(s)\bigr)^2 + g(s) \bigr] \bigl(1+(t-s)^{-\frac{1}{\beta}}\bigr)e^{-\nu(t-s)} \de s  \,,
\end{multline*}
for a big constant $C>0$ depending ultimately only on $M$, which in turn yields, recalling \eqref{pfp15},
\begin{equation*}
\big\| \tilde{y}^1 - \tilde{y}^2 \big\|_{\mathcal{X}_1} \leq C(\delta+\delta_0) \Bigl( \|\Lambda^1-\Lambda^2\|_{\mathcal{X}_2} + \|y^1-y^2\|_{\mathcal{X}_1} \Bigr) .
\end{equation*}
Starting from the inequality
\begin{equation*}
\begin{split}
\big| \tilde{\Lambda}^1(t) - \tilde{\Lambda}^2(t) \big|
& \leq \frac{1-2^\beta}{4} \int_0^t \big| \Lambda^1(s)-\Lambda^2(s)\big|  \big| D^\beta_\infty S(t-s)( P(y^1(s)) ) \big| \de s \\
& + \frac{1-2^\beta}{4} \int_0^t |\Lambda^2(s)| \big| D^\beta_\infty S(t-s)( P(y^1(s))-P(y^2(s)) )  \big| \de s \\
& + \frac{1-2^\beta}{4} \int_0^t \big| D^\beta_\infty S(t-s) \bigl( h(y^1(s),A_{\Lambda^1}(s)) - h(y^2(s),A_{\Lambda^2}(s)) \bigr)\big| \de s
\end{split}
\end{equation*}
a completely similar argument shows that
\begin{equation*}
\big\| \tilde{\Lambda}^1 - \tilde{\Lambda}^2 \big\|_{\mathcal{X}_1} \leq C(\delta+\delta_0) \Bigl( \|\Lambda^1-\Lambda^2\|_{\mathcal{X}_2} + \|y^1-y^2\|_{\mathcal{X}_1} \Bigr) .
\end{equation*}
Therefore the map $\mathcal{T}$ is a contraction in the space $\mathcal{X}$, provided that $\delta$ and $\delta_0$ are small enough (depending on $M$).

\smallskip\noindent\textit{Step 3: conclusion.}
By Banach's Fixed Point Theorem we have obtained the existence of a fixed point $(y,\Lambda)\in\mathcal{X}$ for the map $\mathcal{T}$; that is, denoting by $A(t):=A^0+\int_0^t\Lambda(s)\de s$, the maps $t\mapsto A(t)$, $t\mapsto y(t)$ satisfy the equations \eqref{fp5} and \eqref{fp7}, with $A(0)=A^0$, $y(0)=y^0$.

We now want to show that this is also a solution to the starting equation \eqref{fp1}. Notice first that, in view of \eqref{linear12}, we can rewrite \eqref{fp5} as
\begin{multline*}
\frac{\de A}{\de t}
= - \lim_{n\to\infty} \biggl[
\mathscr{L}_n\bigl(S(t)(y^0)\bigr) + \int_0^t \frac{\de A}{\de t}(s) \mathscr{L}_n\bigl( S(t-s)( P(y(s)) ) \bigr) \de s \\
+ \int_0^t \mathscr{L}_n\bigl( S(t-s)(h(y(s),A(s))) \bigr) \de s \biggr]
\end{multline*}
(where we can pass the limit under the integral sign in view of the uniform estimate \eqref{linear13} and the bounds on $y$ and $\frac{\de A}{\de t}$ in the space $\mathcal{X}$). By integrating this equation in $(0,t)$ we have (exchanging once more limit and integrals, thanks to the uniform estimates)
\begin{equation*}
\begin{split}
A(t) - A(0)
& = - \lim_{n\to\infty}\biggl[ \int_0^t \mathscr{L}_n\bigl(S(s)(y^0)\bigr)\de s
+ \int_0^t \de s \int_0^s \frac{\de A}{\de\xi}(\xi) \mathscr{L}_n\bigl( S(s-\xi)( P(y(\xi)) ) \bigr) \de\xi \\
& \qquad\qquad + \int_0^t \de s \int_0^s \mathscr{L}_n\bigl( S(s-\xi)(h(y(\xi),A(\xi))) \bigr) \de \xi \biggr] \\
& = - \lim_{n\to\infty}\biggl[ \int_0^t \frac{\de}{\de s} \bigl(S_n(s)(y^0)\bigr)\de s
+ \int_0^t \de s \int_0^s \frac{\de A}{\de\xi}(\xi) \frac{\de}{\de s}\bigl( S_n(s-\xi)( P(y(\xi)) ) \bigr) \de\xi \\
& \qquad\qquad + \int_0^t \de s \int_0^s \frac{\de}{\de s}\bigl( S_n(s-\xi)(h(y(\xi),A(\xi))) \bigr) \de \xi \biggr] \\
& = \lim_{n\to\infty}\biggl[ S_n(0)(y^0) - S_n(t)(y^0) \\
& \qquad\qquad + \int_0^t \frac{\de A}{\de\xi}(\xi) \Bigl( S_n(0)( P(y(\xi)) ) - S_n(t-\xi)( P(y(\xi)) ) \Bigr) \de\xi \\
& \qquad\qquad + \int_0^t \Bigl( S_n(0)(h(y(\xi),A(\xi))) - S_n(t-\xi)(h(y(\xi),A(\xi))) \Bigr) \de \xi \biggr] \\
& = - S_\infty(t)(y^0) - \int_0^t \frac{\de A}{\de\xi}(\xi)S_\infty(t-\xi)( P(y(\xi)) )\de\xi \\
& \qquad\qquad - \int_0^t S_\infty(t-\xi)(h(y(\xi),A(\xi)))\de\xi \,.
\end{split}
\end{equation*}
This is exactly the identity \eqref{fp6}. Finally, adding \eqref{fp6} to \eqref{fp7}, we conclude that the pair $(y,A)$ satisfies \eqref{fp3}, which is in turn equivalent to \eqref{fp1}.

We eventually show that $\lim_{t\to\infty}A(t)=A_M$. Denoting by $M(t)$ the mass of the stationary solution $g_p(\cdot;A(t))$, that is, $M(t)=\sum_{n=-\infty}^\infty 2^na_n(A(t))$, we have in view of the mass conservation property \eqref{dirac6}
\begin{equation*}
M = \sum_{n=-\infty}^\infty 2^n a_n(A(t))\bigl( 1+2^n y_n(t) \bigr) = M(t) + \sum_{n=-\infty}^\infty 2^{2n}a_n(A(t))y_n(t),
\end{equation*}
that is, using \eqref{fp10} and the asymptotic properties \eqref{peak5},
\begin{equation*}
|M-M(t)| \leq \biggl( \sum_{n=-\infty}^0 2^{n}a_n(A(t)) + \sum_{n=1}^\infty 2^{(2-\beta)n}a_n(A(t))\biggr) \|y(t)\|_\beta
\leq C (1+t^{-\frac{\beta-1}{\beta}})e^{-\frac{\nu}{2}t}
\end{equation*}
for a uniform constant $C$. Therefore $M(t)$ converges exponentially to $M$ as $t\to\infty$. As the parameter $A$ depends continuously on the mass of the corresponding stationary solution (see the last part of the proof of Proposition~\ref{prop:stationary}), the claim follows.
\end{proof}

The following technical lemma was instrumental in the proof of Proposition~\ref{prop:fixedpoint}.

\begin{lemma} \label{lem:fixedpointh}
Let $y\in\mathcal{Y}_\beta$, $A>0$, and let $h(y,A)$ be the sequence defined by \eqref{fp2}.
Assume also that $A\geq \frac12A_M$. Then there exists a constant $C_2$, depending on $M$, such that
\begin{equation} \label{pfp30}
\|h(y,A)\|_{\beta-1} \leq C_2 \|y\|_{\beta}^2 + C_2 |A_M-A|\|y\|_\beta \,.
\end{equation}
Furthermore, for every $y^1,y^2\in\mathcal{Y}_\beta$ and $A^1,A^2>\frac12 A_M$ we have
\begin{equation} \label{pfp31}
\begin{split}
\| h(y^1,A^1)-h(y^2,A^2) \|_{\beta-1} &\leq C_2\|y^1-y^2\|_\beta \Bigl( \max\{ \|y^1\|_\beta, \|y^2\|_\beta\} + |A_M-A^1| \Bigr) \\
& \qquad + C_2|A^1-A^2| \Bigl( \|y^2\|_\beta^2 + \|y^2\|_\beta \Bigr) \,.
\end{split}
\end{equation}
\end{lemma}

\begin{proof}
Along the proof, the symbol $\lesssim$ will be used for inequalities up to constants which can depend only on the properties of the kernels and on $M$. We first notice that, in view of the explicit expression \eqref{proofstat6} of the coefficients $\alpha_n$, we have
\begin{equation} \label{pfp32}
\begin{split}
\big|\alpha_n(A_M)-\alpha_n(A)\big|
& = \big| e^{-A_M2^n}-e^{-A2^n}\big| \exp\biggl(-2^n\sum_{j=n+1}^\infty 2^{-j}\ln(\theta_{j-1})\biggr) \\
& \lesssim 2^n e^{-(A_M\wedge A)2^n}|A_M-A| \\
& \leq 2^n e^{-\frac12 A_M2^n}|A_M-A| \,.
\end{split}
\end{equation}
Similarly if $A^1,A^2>\frac12 A_M$
\begin{equation}\label{pfp33}
\big|\alpha_n(A^1)-\alpha_n(A^2)\big| \lesssim 2^n e^{-\frac12 A_M2^n}|A^1-A^2| \,.
\end{equation}

We now show \eqref{pfp30}. We first consider the case $n\leq0$. By using the definition of $h_n$, the estimate \eqref{pfp32}, and the asymptotics \eqref{kernel5bis}, \eqref{proofstat5} of $\gamma(2^n)$ and $\alpha_n$ as $n\to-\infty$, we find
\begin{equation*}
|h_n(y,A)| \lesssim 2^{-n}\|y\|_\beta^2 + |A_M-A|\|y\|_\beta \,.
\end{equation*}
For $n>0$, using \eqref{pfp32}, the asymptotics \eqref{kernel5} of $\gamma(2^n)$, and the asymptotics \eqref{proofstat7} of $\alpha_n$ as $n\to\infty$, we obtain
\begin{equation*}
\begin{split}
|h_n(y,A)|
& \lesssim 2^{-(\beta-1)n}\|y\|_\beta^2 + 2^ne^{-\frac12 A_M2^n}|A_M-A|\|y\|_\beta \\
& \lesssim 2^{-(\beta-1)n}\|y\|_\beta^2 + 2^{-(\beta-1)n}|A_M-A|\|y\|_\beta \,.
\end{split}
\end{equation*}
Then \eqref{pfp30} follows combining the previous estimates.

We next prove \eqref{pfp31}. For $n\leq0$, using the definition of $h$, the estimate \eqref{pfp33}, and the asymptotics of $\gamma(2^n)$ and $\alpha_n$, we find
\begin{equation*}
\begin{split}
|h_n(y^1,A^1)&-h_n(y^2,A^2)| \\
& \lesssim 2^{-n} \|y^1-y^2\|_\beta\|y^1+y^2\|_\beta
+ 2^{-n} |\alpha_n(A^1)-\alpha_n(A^2)| \|y^2\|_\beta^2 \\
& \qquad + 2^{-n}|\alpha_n(A_M)-\alpha_n(A^1)| \|y^1-y^2\|_\beta
+ 2^{-n}|\alpha_n(A^1)-\alpha_n(A^2) \|y^2\|_\beta \\
& \lesssim 2^{-n} \|y^1-y^2\|_\beta\|y^1+y^2\|_\beta + e^{-\frac12 A_M 2^n}|A^1-A^2| \|y^2\|_\beta^2 \\
& \qquad + e^{-\frac12 A_M 2^n}|A_M-A^1| \|y^1-y^2\|_\beta + e^{-\frac12 A_M 2^n}|A^1-A^2| \|y^2\|_\beta \,.
\end{split}
\end{equation*}
For $n>0$ we have instead
\begin{equation*}
\begin{split}
|h_n(y^1,A^1)&-h_n(y^2,A^2)| \\
& \lesssim 2^{-(\beta-1)n} \|y^1-y^2\|_\beta\|y^1+y^2\|_\beta
+ 2^{-(\beta-1)n} |\alpha_n(A^1)-\alpha_n(A^2)| \|y^2\|_\beta^2 \\
& \qquad + |\alpha_n(A_M)-\alpha_n(A^1)| \|y^1-y^2\|_\beta
+ |\alpha_n(A^1)-\alpha_n(A^2) \|y^2\|_\beta \\
& \lesssim 2^{-(\beta-1)n} \|y^1-y^2\|_\beta\|y^1+y^2\|_\beta + 2^ne^{-\frac12 A_M 2^n}|A^1-A^2| \|y^2\|_\beta^2 \\
& \qquad + 2^ne^{-\frac12 A_M 2^n}|A_M-A^1| \|y^1-y^2\|_\beta + 2^ne^{-\frac12 A_M 2^n}|A^1-A^2| \|y^2\|_\beta \,.
\end{split}
\end{equation*}
The two estimates combined give \eqref{pfp31}.
\end{proof}

\begin{remark} \label{rm:ass2}
The strategy of the proof of Proposition~\ref{prop:fixedpoint} requires the assumption $\beta<2$, as otherwise the first integral in \eqref{pfp15} would be divergent. The same approach could be followed also in the case $\beta\geq2$, but under the stronger assumption on the initial datum $\|y^0\|_\beta\leq\delta_0$. We however expect that the same statement can be proved also for $\beta\geq2$.
\end{remark}

The proof of Theorem~\ref{thm:dirac} follows now directly from Proposition~\ref{prop:fixedpoint}.

\begin{proof}[Proof of Theorem~\ref{thm:dirac}]
It is sufficient to reformulate the statement of Proposition~\ref{prop:fixedpoint} in terms of the variable $\e_n=2^ny_n$.
\end{proof}


\appendix
\section{Proof of the regularity result for the linearized problem} \label{sect:appendix}

This section is entirely devoted to the proof of the regularity result Theorem~\ref{thm:linear} on the linearized problem \eqref{linear1}. The analysis will be performed in two steps. We first show in Lemma~\ref{lem:linear1} the global well-posedness of \eqref{linear1}, via maximum principle arguments, for an initial datum  $y^0\in\mathcal{Y}_\theta$. This allows to define the semigroup $t\mapsto S(t)(y^0)$, see \eqref{linear10}. In the same lemma we prove that the solution converges uniformly to a constant, as $t\to\infty$, in regions $n\in(-\infty,n_0)$ for $n_0$ arbitrarily large. In a second step (Lemma~\ref{lem:linear2}) we analyze the behaviour of the solution for large values $n\to\infty$.

We observe that all the constants in the statements below will depend of course also on the properties of the coagulation and fragmentation kernels; however we will not mention this dependence explicitly, as they are fixed throughout the paper.

\begin{lemma} \label{lem:linear1}
Let $\theta\geq-1$ and let $y^0=\{y_n^0\}_{n\in\Z}\in\mathcal{Y}_\theta$ be a given initial datum. Then there exists a unique solution $t\mapsto y(t)=\{y_n(t)\}_{n\in\Z}$ to \eqref{linear1}, with $y(0)=y^0$, in the space $\mathcal{Y}$ for $\theta\geq0$ and in the space $\mathcal{Y}_\theta$ if $\theta<0$.

Moreover, for every sufficiently large $n_0\in\N$ there exists a constant $C_{n_0}>0$ (depending only $n_0$ and $M$) such that
\begin{equation} \label{linear5}
|y_n(t)-\bar{m}| \leq C_{n_0}\|y^0\|_\theta(2^{-n}+1)e^{-\nu t}  \qquad\text{for all $n\leq n_0$,}
\end{equation}
where $\nu>0$ depends only on $M$ and
\begin{equation} \label{linear6}
\bar{m} := \frac{\sum_{n=-\infty}^\infty 2^{2n}a_n(A_M)y_n^0}{\sum_{n=-\infty}^\infty 2^{2n}a_n(A_M)} \,.
\end{equation}
\end{lemma}

The bound $\theta\geq-1$ in the statement is artificial and not needed in the proof; we include it in order to obtain uniform constants with respect to $\theta$.
We remark for later use that, if $y^0\in\mathcal{Y}_\theta$, then thanks to the asymptotics \eqref{peak5} of $a_n$
\begin{equation} \label{linear6bis}
|\bar{m}| \leq C_0 \|y^0\|_\theta
\end{equation}
for a constant $C_0$ depending only on $M$.

\begin{proof}[Proof of Lemma~\ref{lem:linear1}]
Along the proof we simplify the notation by writing $a_n=a_n(A_M)$, $\alpha_n=\alpha_n(A_M)$.
We divide the proof of the lemma into three steps, first showing the well-posedness of \eqref{linear1} in $\mathcal{Y}_\theta$, then proving the estimate \eqref{linear5} in any bounded region $n\in[-n_0,n_0]$, and eventually proving \eqref{linear5} for $n<-n_0$.

\smallskip\noindent\textit{Step 1.}
We show the global well-posedness of \eqref{linear1}. Consider first the case $\theta\geq0$: in this case the sequence $y^0$ is bounded as $n\to\infty$. It is straightforward to check that the function
\begin{equation} \label{plinear9}
\bar{y}_n(t) :=
\begin{cases}
2^{-n}e^{\mu t} & n\leq0,\\
e^{\mu t} & n>0
\end{cases}
\end{equation}
is a supersolution of the problem \eqref{linear1}, that is $\frac{\de\bar{y}_n}{\de t} - \mathscr{L}_n(\bar{y})>0$, provided that 
\begin{equation*}
\mu > \max\Bigl\{ \frac{\gamma(1)}{4}, \, \sup_{n<0} \frac{\gamma(2^n)}{4}(1-\sigma_n/2) \Bigr\}
\end{equation*}
(such a value of $\mu$ exists thanks to the asymptotics \eqref{linear3} of $\sigma_n$). The existence of a solution to \eqref{linear1} can then be proved by a standard truncation argument: for $N\in\N$, one first constructs a solution $y^N(t) = \{ y_n^N(t) \}_{n\in\Z}$ to the finite set of equations
$$
\frac{\de y^N_n}{\de t}(t) = \mathscr{L}_n(y^N(t)), \qquad |n|\leq N,
$$
with boundary values $y_n^N(t)\equiv0$ for $|n|>N$, and initial datum $y_n^N(0)=y_n^0$ for $|n|\leq N$. By comparison with the supersolution \eqref{plinear9}, the Maximum Principle yields the estimate
\begin{equation*}
|y_n^N(t)|
\leq \|y^0\|_{\theta}\bar{y}_n(t)
\leq \|y^0\|_{\theta} (2^{-n}+1) e^{\mu t}
\qquad\text{for all $t>0$ and $|n|\leq N$,}
\end{equation*}
which is in particular uniform in $N$. Letting $N\to\infty$, by a standard compactness argument we obtain a solution $y(t)=\{y_n(t)\}_{n\in\Z}$ to \eqref{linear1} with $y(0)=y^0$, satisfying the estimate
\begin{equation} \label{plinear10}
\|y(t)\|_{0} \leq 2 \|y^0\|_{\theta}e^{\mu t} \qquad\text{for every $t>0$} \qquad\text{(if $\theta\geq0$).}
\end{equation}

This argument has to be modified in the case of an initial datum $y^0\in\mathcal{Y}_\theta$ with $\theta<0$ (unbounded as $n\to\infty$), since we are not allowed to compare with the supersolution \eqref{plinear9}. However, one can check that in this case the sequence
\begin{equation*}
\bar{y}_n(t) :=
\begin{cases}
2^{-n}e^{\mu t} & n\leq0,\\
2^{-\theta n}e^{\mu t} & n>0
\end{cases}
\end{equation*}
is a supersolution of the problem \eqref{linear1}, provided that $\mu$ is large enough (thanks to the fast decay of $\sigma_n$ as $n\to\infty$, see \eqref{linear3}); hence by repeating the previous argument we obtain a solution $y(t)=\{y_n(t)\}_{n\in\Z}$ to \eqref{linear1} with $y(0)=y^0$, satisfying the estimate
\begin{equation} \label{plinear10b}
\|y(t)\|_{\theta} \leq 2 \|y^0\|_{\theta}e^{\mu t} \qquad\text{for every $t>0$} \qquad\text{(if $\theta<0$).}
\end{equation}

The uniqueness of the solution can be also obtained by a maximum principle argument. As before, we first consider the case $\theta\geq0$: in this case we obtain uniqueness in the space $\mathcal{Y}$. Indeed, assume that $y(t)=\{y_n(t)\}_{n\in\Z}$ is a solution to \eqref{linear1}, with $y(t)\in\mathcal{Y}$ for every $t>0$, and $y(0)=0$. A direct computation shows that the sequence
\begin{equation*}
\tilde{y}_n(t) :=
\begin{cases}
4^{-n}e^{\tilde{\mu}t} & n < 0,\\
4^{n}e^{\tilde{\mu}t} & n \geq 0
\end{cases}
\end{equation*}
is a supersolution for \eqref{linear1}, provided $\tilde{\mu}$ is large enough. For every given $T>0$ and $\e>0$ we can find $\overline{N}=\overline{N}(\e,\sup_{t\in[0,T]}\|y(t)\|_\mathcal{Y})$ such that $|y_{\pm N}(t)| \leq \e\tilde{y}_{\pm N}(t)$ for all $N\geq\overline{N}$ and $t\in[0,T]$. By applying the maximum principle in the bounded region $(n,t)\in[-N,N]\times[0,T]$ we obtain
\begin{equation*}
-\e\tilde{y}_n(t) \leq y_n(t) \leq \e\tilde{y}_n(t) \qquad\text{for all $n\in[-N,N]$ and $t\in[0,T]$.}
\end{equation*}
By letting first $N\to\infty$, and then $\e\to0$, $T\to\infty$, we can conclude that $y_n(t)=0$ for all $n\in\Z$ and $t>0$.

In the case of an initial datum $y^0\in\mathcal{Y}_\theta$ with $\theta<0$, one can repeat the previous argument with the modified supersolution
\begin{equation*}
\tilde{y}_n(t) :=
\begin{cases}
2^{-(1-\theta)n}e^{\tilde{\mu}t} & n < 0,\\
2^{(1-\theta)n}e^{\tilde{\mu}t} & n \geq 0
\end{cases}
\end{equation*}
(for $\tilde{\mu}$ large enough), and obtain uniqueness in the space $\mathcal{Y}_\theta$.
This completes the proof of the well-posedness of \eqref{linear1}.

\smallskip\noindent\textit{Step 2.}
Recalling the definition \eqref{linear2} of $\sigma_n$ and \eqref{peak8}, it is straightforward to check that \eqref{linear1} can be written in the form
\begin{equation} \label{plinear11}
\frac{\de y_n}{\de t} = \frac{1}{2^{2n}a_n} D^-_n\Bigl( \bigl\{ 2^{2k}\gamma(2^{k+1})a_{k+1} D_k^+(y) \bigr\}_k \Bigr).
\end{equation}
Subtracting the quantity $\bar{m}$ defined in \eqref{linear6}, we have
\begin{equation*}
2^{2n}a_n \frac{\de}{\de t}(y_n-\bar{m}) =  D^-_n\Bigl( \bigl\{ 2^{2k}\gamma(2^{k+1})a_{k+1} D^+_k(y) \bigr\}_k \Bigr),
\end{equation*}
and in turn, multiplying by $y_n-\bar{m}$ and summing over $n$, we end up with
\begin{align*}
\frac{\de}{\de t} \biggl( \sum_{n=-\infty}^\infty 2^{2n}a_n&(y_n-\bar{m})^2 \biggr) \\
& = 2\sum_{n=-\infty}^\infty \Bigl( 2^{2n}\gamma(2^{n+1})a_{n+1} D_n^+(y) - 2^{2n-2}\gamma(2^{n})a_{n} D^+_{n-1}(y) \Bigr)(y_n-\bar{m}) \\
& = 2\sum_{n=-\infty}^\infty \Bigl( 2^{2n}\gamma(2^{n+1})a_{n+1} D^+_n(y) \Bigr)(y_n-y_{n+1}) \\
& = - 2\sum_{n=-\infty}^\infty 2^{2n} \gamma(2^{n+1})a_{n+1}(D^+_n(y))^2 \,.
\end{align*}
By Lemma~\ref{lem:poincare} below we obtain
\begin{equation*}
\frac{\de}{\de t} \biggl( \sum_{n=-\infty}^\infty 2^{2n}a_n(y_n-\bar{m})^2 \biggr)
\leq -\frac{2}{c_0} \sum_{n=-\infty}^\infty 2^{2n}a_n(y_n-\bar{m})^2  \,,
\end{equation*}
which in turn yields
\begin{equation} \label{plinear6}
\|y(t)-\bar{m}\|^2_{\ell^2(\Z;2^na_n^{1/2})} \leq e^{-\frac{2}{c_0}t}\|y^0-\bar{m}\|^2_{\ell^2(\Z;2^na_n^{1/2})} \,.
\end{equation}
Notice now that by \eqref{linear6bis}
\begin{equation} \label{plinear6bis}
\|y^0-\bar{m}\|^2_{\ell^2(\Z;2^na_n^{1/2})}
= \sum_{n=-\infty}^\infty 2^{2n}a_n(y_n^0-\bar{m})^2
\leq C_0'\|y^0\|_\theta^2
\end{equation}
for another constant $C_0'>0$ depending only on $M$. It is immediately seen that \eqref{plinear6} and \eqref{plinear6bis} imply the uniform convergence of $y_n(t)$ to the constant $\bar{m}$ as $t\to\infty$, for $n$ in any compact region: for every $n_0\in\N$ there exists a constant $C_{n_0}'$ (depending on $n_0$ and $M$) such that for every $t>0$
\begin{equation} \label{plinear7}
\sup_{n\in[-n_0,n_0]}|y_n(t)-\bar{m}| \leq C_{n_0}' e^{-\frac{t}{c_0}}\|y^0\|_{\theta} \,.
\end{equation}

\smallskip\noindent\textit{Step 3.}
It only remains to control the region $n<-n_0$. Let $T>0$ and $\e>0$ be fixed. Consider the sequence
\begin{equation*}
z_n(t) := C 2^{-n}e^{-\nu t} + \e 4^{-n},
\end{equation*}
where $C>0$ and $\nu>0$ are constant to be fixed later. We first observe that by as straightforward computation
\begin{equation*}
\frac{\de z_n}{\de t} - \mathscr{L}_n(z) = C2^{-n}e^{-\nu t} \Bigl( -\nu - \frac{\gamma(2^n)}{4}(1-\sigma_n/2) \Bigr) - \e\frac{\gamma(2^n)}{4^{n+1}}\Bigl(3-\frac34\sigma_n\Bigr).
\end{equation*}
Recalling that $\gamma(2^n)\to\gamma_0>0$ and $\sigma_n\to8$ as $n\to-\infty$, by choosing $\nu<\frac34\gamma_0$ we obtain that $z_n$ is a supersolution for \eqref{linear1} in the region $n\in(-\infty,-n_0]$, for every sufficiently large $n_0$. Furthermore, for $t=0$ (by \eqref{linear6bis})
\begin{equation*}
|y_n(0)-\bar{m}| \leq \|y^0\|_{\theta}(2^{-n}+C_0) \leq C2^{-n} \leq z_n(0) \qquad\text{for all }n\leq -n_0,
\end{equation*}
provided that we choose $C>\|y^0\|_{\theta}(1+C_0)$. By \eqref{plinear7}, for $n=-n_0$ we have
\begin{equation*}
|y_{-n_0}(t)-\bar{m}| \leq C_{n_0}' e^{-\frac{t}{c_0}}\|y^0\|_{\theta} \leq z_{-n_0}(t) \qquad\text{for every }t>0,
\end{equation*}
if we choose $\nu<\frac{1}{c_0}$ and $C>2^{-n_0}C_{n_0}' \|y^0\|_{\theta}$.
Finally, by \eqref{linear6bis} and \eqref{plinear10}--\eqref{plinear10b} we can choose $n_1>n_0$ sufficiently large, depending on $\e$ and $T$, such that
\begin{equation*}
|y_{-n_1}(t)-\bar{m}| \leq 2^{n_1+1}\|y^0\|_{\theta}e^{\mu t} + C_0\|y^0\|_\theta \leq \e4^{n_1} \leq z_{-n_1}(t) \qquad\text{for every }t\in[0,T].
\end{equation*}

Therefore, with the choices
\begin{equation}
C > \max\bigl\{ (1+C_0), 2^{-n_0}C_{n_0}' \bigr\}\|y^0\|_{\theta},
\qquad
\nu <\min\Bigl\{ \frac34\gamma_0, \frac{1}{c_0} \Bigr\},
\end{equation}
we can apply the Maximum Principle in the compact region $(n,t)\in[-n_1,-n_0]\times[0,T]$:
\begin{equation*}
|y_n(t)-\bar{m}| \leq z_n(t) \qquad\text{for all $n\in[-n_1,-n_0]$ and $t\in[0,T]$.}
\end{equation*}
Letting firstly $n_1\to\infty$, and then $\e\to0$, $T\to\infty$, the previous argument shows that
\begin{equation} \label{plinear8}
|y_n(t)-\bar{m}| \leq C2^{-n}e^{-\nu t} \qquad\text{for all $n\leq -n_0$ and $t>0$.}
\end{equation}
The estimate \eqref{linear5} follows by combining \eqref{plinear7} and \eqref{plinear8}.
\end{proof}

The following discrete Poincar\'e-type inequality is used in the proof of Lemma~\ref{lem:linear1}.

\begin{lemma}\label{lem:poincare}
With the notation introduced in Lemma~\ref{lem:linear1}, there exists a constant $c_0>0$ (depending on $M$) such that for every $t>0$
\begin{equation} \label{poincare}
\sum_{n=-\infty}^\infty 2^{2n}a_n(y_n(t)-\bar{m})^2 \leq c_0 \sum_{n=-\infty}^\infty 2^{2n} \gamma(2^{n+1})a_{n+1}(D^+_n(y(t)))^2 \,.
\end{equation}
\end{lemma}

\begin{proof}
We claim that there exists a constant $c_1>0$ such that
\begin{equation} \label{proofpoincare1}
\sum_{n=-\infty}^\infty 2^{2n}a_n(y_n(t)-y_0(t))^2 \leq c_1 \sum_{n=-\infty}^\infty 2^{2n} \gamma(2^{n+1})a_{n+1}(D^+_n(y(t)))^2 \,.
\end{equation}
Notice that the conclusion of the lemma will follow easily from \eqref{proofpoincare1}: indeed, we can write
\begin{equation*}
\bar{m} = \frac{\sum_{n=-\infty}^\infty 2^{2n}a_n y_n(t)}{\sum_{n=-\infty}^\infty 2^{2n}a_n} \,,
\end{equation*}
since the right-hand side is actually independent of $t$, as can be easily checked by using the equation \eqref{linear1}; we then have
\begin{equation*}
\begin{split}
\biggl(\sum_{n=-\infty}^\infty 2^{2n}a_n\biggr) (\bar{m}-y_0(t))
& = \sum_{n=-\infty}^\infty 2^{2n}a_n(y_n(t)-y_0(t)) \\
& \leq \biggl(\sum_{n=-\infty}^\infty 2^{2n}a_n(y_n(t)-y_0(t))^2\biggr)^\frac12 \biggl(\sum_{n=-\infty}^\infty 2^{2n}a_n\biggr)^\frac12,
\end{split}
\end{equation*}
from which it follows that
\begin{equation*}
\sum_{n=-\infty}^\infty 2^{2n}a_n (\bar{m}-y_0(t))^2 \leq \sum_{n=-\infty}^\infty 2^{2n}a_n(y_n(t)-y_0(t))^2.
\end{equation*}
Therefore, assuming that \eqref{proofpoincare1} holds,
\begin{equation*}
\begin{split}
\sum_{n=-\infty}^\infty 2^{2n}a_n(y_n(t)-\bar{m})^2
& \leq 2\sum_{n=-\infty}^\infty 2^{2n}a_n(y_n(t)-y_0(t))^2 + 2\sum_{n=-\infty}^\infty 2^{2n}a_n(y_0(t)-\bar{m})^2 \\
& \leq 4c_1 \sum_{n=-\infty}^\infty 2^{2n} \gamma(2^{n+1})a_{n+1}(D^+_n(y(t)))^2 \,,
\end{split}
\end{equation*}
which gives \eqref{poincare}. We are then left with the proof of \eqref{proofpoincare1}, which we show in two steps. In the following we omit the dependence on the variable $t$, which is fixed.

\smallskip\noindent\textit{Step 1: $n\geq1$.} By writing $y_n=y_0+\sum_{k=0}^{n-1}D^+_k(y)$ we find
\begin{equation} \label{proofpoincare2}
\begin{split}
\sum_{n=1}^\infty 2^{2n}a_n(y_n-y_0)^2
& = \sum_{n=1}^\infty 2^{2n}a_n\biggl(\sum_{k=0}^{n-1}D^+_k(y)\biggr)^2
\leq \sum_{n=1}^\infty 2^{2n}a_n n\sum_{k=0}^{n-1}|D^+_k(y)|^2 \\
& = \sum_{k=0}^\infty \biggl(\sum_{n=k+1}^{\infty}2^{2n}a_n n\biggr)|D^+_k(y)|^2 \,.
\end{split}
\end{equation}
Recalling \eqref{peak8} and \eqref{proofstat6}, for $n\geq k+2$ we have
\begin{equation*}
\begin{split}
a_n
& = \biggl(\prod_{m=k+1}^{n-1}2\alpha_m\biggr)a_{k+1} \\
& = 2^{n-k-1}\exp\biggl(-A_M\sum_{m=k+1}^{n-1}2^m\biggr)\exp\biggl(-\sum_{m=k+1}^{n-1}2^{m}\sum_{j=m+1}^\infty2^{-j}\ln(\theta_{j-1})\biggr)a_{k+1} \\
& \leq 2^{n-k-1}e^{-A_M(2^n-2^{k+1})}\exp\Bigl((n-k-1)\sup_{j\in\Z}|\ln(\theta_j)|\Bigr)a_{k+1} \\
& \leq 2^{c(n-k-1)}e^{-A_M2^n(1-2^{k+1-n})} a_{k+1}
\end{split}
\end{equation*}
for a uniform constant $c$ (notice that the coefficients $\theta_j$ depend only on the coagulation and fragmentation kernels, and the asymptotics \eqref{peak9} yields the uniform boundedness of $|\ln(\theta_j)|$). By inserting this estimate in \eqref{proofpoincare2} we find
\begin{equation*}
\begin{split}
\sum_{n=1}^\infty 2^{2n}a_n(y_n-y_0)^2
& \leq \sum_{k=0}^\infty  \biggl(\sum_{n=k+1}^{\infty}2^{2n}n 2^{c(n-k-1)}e^{-A_M2^n(1-2^{k+1-n})} \biggr) a_{k+1} |D^+_k(y)|^2 \\
& \leq \sum_{k=0}^\infty  \biggl(\sum_{n=k+1}^{\infty}2^{2(n-k)}n 2^{c(n-k-1)}e^{-A_M2^n(1-2^{k+1-n})} \biggr) 2^{2k}a_{k+1} |D^+_k(y)|^2 \,.
\end{split}
\end{equation*}
The strict positivity of $\gamma(\xi)$, together with the asymptotics \eqref{kernel5}, yields the existence of a constant $c_1$ (depending on $M$) such that
\begin{equation*}
\sum_{n=1}^\infty 2^{2n}a_n(y_n-y_0)^2 \leq c_1 \sum_{n=0}^\infty 2^{2n} \gamma(2^{n+1})a_{n+1}(D^+_n(y))^2 \,.
\end{equation*}

\smallskip\noindent\textit{Step 2: $n\leq-1$.} Introduce variables $z_n:=2^na_n^{1/2}(y_n-y_0)$, $x_n:=z_{n+1}-2(\frac{a_{n+1}}{a_{n}})^{1/2}z_{n}$, so that
\begin{equation*}
D^+_n(y)=D^+_n(y-y_0) = 2^{-(n+1)}a_{n+1}^{-1/2}z_{n+1} - 2^{-n}a_{n}^{-1/2}z_{n}= \frac{x_n}{2^{n+1}a_{n+1}^{1/2}} \,.
\end{equation*}
In these variables the claim amounts to show that
\begin{equation} \label{proofpoincare5}
\sum_{n=-\infty}^{-1} |z_n|^2 \leq c_1\sum_{n=-\infty}^{-1} \gamma(2^{n+1})|x_n|^2 \,.
\end{equation}
By using the definition of $x_n$ and recalling \eqref{peak8} we find the recurrence formula
\begin{equation} \label{proofpoincare6}
z_n = - \sum_{k=0}^{-n-1} (2\sqrt{2})^{-(k+1)}\biggl(\prod_{j=0}^{k}\alpha^{-1/2}_{n+j}\biggr) x_{n+k} \qquad\text{for all $n\leq-1$.}
\end{equation}
Now for all $n\leq-1$ and $k\in\{0,\ldots,-n-1\}$ we have by \eqref{proofstat6}
\begin{align*}
\prod_{j=0}^{k}\alpha^{-1/2}_{n+j}
& = \exp\biggl(A_M\sum_{j=0}^k 2^{n+j-1} \biggr) \exp\biggl( \sum_{j=0}^k 2^{n+j-1}\sum_{\ell=n+j+1}^\infty 2^{-\ell}\ln(\theta_{\ell-1}) \biggr) \\
& \leq e^{A_M}\exp\biggl(\sum_{\ell=-\infty}^\infty 2^{-\ell}|\ln(\theta_{\ell-1})|\biggr) =: \bar{c}
\end{align*}
(recall that the series converges thanks to the asymptotics \eqref{peak9} of the coefficients $\theta_\ell$). Combining this estimate with \eqref{proofpoincare6} we obtain
\begin{align} \label{proofpoincare7}
\sum_{n=-\infty}^{-1}|z_n|^2
\leq \sum_{n=-\infty}^{-1} \biggl( \bar{c}\sum_{k=0}^{-n-1} (2\sqrt{2})^{-(k+1)} x_{n+k} \biggr)^2
= \frac{\bar{c}^2}{8} \sum_{n=-\infty}^{-1} \biggl( \sum_{m=n}^{-1} (2\sqrt{2})^{n-m} x_{m} \biggr)^2 \,.
\end{align}
To complete the proof, we use a discrete version of Young's convolution inequality: letting $f_n=\sum_{m=n}^{-1} (2\sqrt{2})^{n-m} x_{m}$ we have
\begin{align*}
\sum_{n=-\infty}^{-1}f_n^2
& = \sum_{n=-\infty}^{-1}\sum_{m=n}^{-1} f_n(2\sqrt{2})^{n-m} x_{m} \\
& \leq \biggl(\sum_{n=-\infty}^{-1}\sum_{m=n}^{-1}f_n^2(2\sqrt{2})^{n-m} \biggr)^\frac12 \biggl(\sum_{n=-\infty}^{-1}\sum_{m=n}^{-1} x_m^2(2\sqrt{2})^{n-m} \biggr)^\frac12 \\
& \leq \biggl(\sum_{n=-\infty}^{-1}f_n^2\biggr)^\frac12 \biggl(\sum_{j=0}^{\infty}(2\sqrt{2})^{-j} \biggr) \biggl(\sum_{n=-\infty}^{-1}x_n^2\biggr)^\frac12 ,
\end{align*}
that is,
\begin{equation} \label{proofpoincare8}
\sum_{n=-\infty}^{-1}f_n^2 \leq 4\sum_{n=-\infty}^{-1}x_n^2 \,.
\end{equation}
By inserting \eqref{proofpoincare8} into \eqref{proofpoincare7} we end up with
\begin{equation*}
\sum_{n=-\infty}^{-1}|z_n|^2  \leq \frac{\bar{c}^2}{2}\sum_{n=-\infty}^{-1}x_n^2 \,.
\end{equation*}
The existence of a constant $c_1$ for which \eqref{proofpoincare5} holds follows now from the strict positivity of $\gamma$ and from \eqref{kernel5bis}.
\end{proof}

In order to study the behaviour of solutions to the linearized equation \eqref{linear1} as $n\to\infty$, in a first approximation we can neglect the term containing the coefficients $\sigma_n$, as its contribution will be negligible for large values of $n$ in view of the fast decay \eqref{linear3}. We will only consider the region $n\geq n_0$, where $n_0\in\N$ is a sufficiently large constant. In particular, for those values we can use the asymptotics \eqref{kernel5}, and we will always assume without loss of generality that
\begin{equation} \label{plinear20}
\gamma(2^n)<\gamma(2^{n+1}),
\qquad
\frac{1}{2} \, 2^{\beta(n-m)} \leq \frac{\gamma(2^n)}{\gamma(2^m)} \leq \frac32 \, 2^{\beta(n-m)}
\qquad
\text{for all $n,m\geq n_0$.}
\end{equation}
We now construct the fundamental solution to the simplified problem without $\sigma_n$.

\begin{lemma} \label{lem:linearfundsol}
Let $n_0\in\N$ be such that \eqref{plinear20} holds. For $\ell\in\Z$, $\ell\geq n_0$, let $\Psi^{(\ell)}_n$ be the solution to the problem
\begin{equation} \label{linearfund1}
\begin{cases}
\frac{\de \Psi_n^{(\ell)}}{\de t} = \frac{\gamma(2^n)}{4}\bigl( \Psi_{n-1}^{(\ell)} - \Psi_n^{(\ell)} \bigr) ,\\
\Psi_n^{(\ell)}(0) = \delta(n-\ell).
\end{cases}
\end{equation}
Then there exists a uniform constant $c_1>0$ such that
\begin{equation} \label{linearfund2}
\big| \Psi_n^{(\ell)}(t) - \Psi_{n+1}^{(\ell)}(t) \big| \leq c_1 2^{-\beta(n-\ell)}e^{-\frac{\gamma(2^\ell)}{4} t} \qquad\text{for all $n\geq\ell\geq n_0$.}
\end{equation}
In particular, there exists the limit $\Psi^{(\ell)}_\infty(t):=\lim_{n\to\infty}\Psi^{(\ell)}_n(t)$, which satisfies
\begin{equation} \label{linearfund3}
\big| \Psi_n^{(\ell)}(t) - \Psi_{\infty}^{(\ell)}(t) \big| \leq c_1 2^{-\beta(n-\ell)}e^{-\frac{\gamma(2^\ell)}{4} t} \qquad\text{for all $n\geq\ell\geq n_0$.}
\end{equation}
Finally, there exists also the limit
\begin{equation} \label{linearfund4}
D^\beta_\infty \Psi^{(\ell)}(t) := \lim_{n\to\infty} 2^{\beta n}\Bigl( \Psi_n^{(\ell)}(t) - \Psi_{\infty}^{(\ell)}(t) \Bigr),
\quad\text{with }\;
|D^\beta_\infty \Psi^{(\ell)}(t)| \leq c_12^{\beta\ell}e^{-\frac{\gamma(2^\ell)}{4}t} \,.
\end{equation}
\end{lemma}

\begin{proof}
We compute the Laplace transform $\widetilde{\Psi}_n^{(\ell)}(z)=\int_0^\infty\Psi_n^{(\ell)}(t)e^{-zt}\de t$ of $\Psi_n^{(\ell)}$ , which solves
\begin{equation*}
z\widetilde{\Psi}_n^{(\ell)}(z) = \frac{\gamma(2^n)}{4}\bigl( \widetilde{\Psi}_{n-1}^{(\ell)}(z) - \widetilde{\Psi}_n^{(\ell)}(z) \bigr) + \delta(n-\ell),
\end{equation*}
and therefore it is explicitly given by the recurrence formula
\begin{equation*}
\widetilde{\Psi}_n^{(\ell)}(z) =
\begin{cases}
0 & n<\ell,\\
\frac{4}{\gamma(2^\ell)} \prod_{k=\ell}^n \bigl( 1 + \frac{4}{\gamma(2^k)}z \bigr)^{-1} & n\geq\ell.
\end{cases}
\end{equation*}
Now $\Psi_n^{(\ell)}$ can be computed using the inverse Laplace transform together with contour integration: all the singularities of $\widetilde{\Psi}_n^{(\ell)}$ are simple poles (indeed $\gamma(2^k)\neq\gamma(2^j)$ for $k,j\geq\ell$, $k\neq j$, by \eqref{plinear20}), located at negative real numbers. Therefore we have
\begin{equation*}
\Psi_n^{(\ell)}(t) = \frac{1}{2\pi i}\int_{-i\infty}^{i\infty} e^{zt}\widetilde{\Psi}_n^{(\ell)}(z)\de z,
\end{equation*}
where the previous integral is a complex integral on the imaginary axis; using Cauchy's Residue Theorem we find (for $t>0$)
\begin{equation} \label{plinear12}
\Psi_n^{(\ell)}(t)=
\begin{cases}
\displaystyle\sum_{k=\ell}^n \frac{\gamma(2^k)}{\gamma(2^\ell)} \prod_{\substack{j=\ell\\j\neq k}}^n \biggl(1-\frac{\gamma(2^k)}{\gamma(2^j)}\biggr)^{-1} e^{-\frac{\gamma(2^k)}{4}t} & n\geq\ell, \\
0 & n<\ell.
\end{cases}
\end{equation}
We remark for later use that for every $n\geq \ell$
\begin{equation} \label{plinear12bis}
\frac{\gamma(2^{\ell})}{4} \int_0^\infty \Psi_n^{(\ell)}(s)\de s = \frac{\gamma(2^{\ell})}{4}\widetilde{\Psi}^{(\ell)}_n(0) = 1 \,.
\end{equation}

We now derive the decay estimates in the statement, using the explicit expression \eqref{plinear12}. We first notice that, thanks to \eqref{plinear20}, we have for all $\ell\leq k \leq n$
\begin{align} \label{plinear13}
\frac{\gamma(2^k)}{\gamma(2^\ell)} \prod_{\substack{j=\ell\\j\neq k}}^n \bigg| 1-\frac{\gamma(2^k)}{\gamma(2^j)} \bigg|^{-1}
& = \frac{\gamma(2^k)}{\gamma(2^\ell)} \biggl(\prod_{j=\ell}^{k-1} \frac{1}{\frac{\gamma(2^k)}{\gamma(2^j)}-1}\biggr) \biggl(\prod_{j=k+1}^n \frac{1}{1-\frac{\gamma(2^k)}{\gamma(2^j)}}\biggr) \nonumber\\
& \leq \frac32 \, 2^{\beta(k-\ell)} \biggl( \prod_{j=\ell}^{k-1} \frac{1}{\frac12 2^{\beta(k-j)}-1} \biggr) \biggl( \prod_{j=k+1}^n \frac{1}{1-\frac32 2^{\beta(k-j)}} \biggr) \nonumber\\
& = \frac32 \, 2^{\beta(k-\ell)} \prod_{m=1}^{k-\ell}2^{-\beta m+1} \biggl(\prod_{m=1}^{k-\ell} \frac{1}{1-2^{-\beta m+1}} \biggr) \biggl(\prod_{m=1}^{n-k} \frac{1}{1-\frac322^{-\beta m}}\biggr) \nonumber\\
& \leq c\, 2^{\beta(k-\ell)} 2^{k-\ell} 2^{-\frac{\beta}{2}(k-\ell)(k-\ell+1)}
\leq c \, 4^{\beta(k-\ell)} 2^{-\frac{\beta}{2}(k-\ell)^2} ,
\end{align}
where $c>0$ is a numerical constant (depending only on $\beta$).
From \eqref{plinear13} it follows that
\begin{equation} \label{plinear14}
|\Psi_n^{(\ell)}(t)| \leq c e^{-\frac{\gamma(2^\ell)}{4}t} \sum_{m=0}^{n-\ell} 4^{\beta m} 2^{-\frac{\beta}{2}m^2} \leq c_1 e^{-\frac{\gamma(2^\ell)}{4}t},
\end{equation}
for another numerical constant $c_1>0$, also depending only on $\beta$. Using \eqref{plinear13} and \eqref{plinear20} we can further estimate the difference of $\Psi_n^{(\ell)}$ and $\Psi_{n+1}^{(\ell)}$, for $n\geq\ell$, as follows:
\begin{equation*}
\begin{split}
\big| \Psi_n^{(\ell)}(t) - \Psi_{n+1}^{(\ell)}(t) \big|
& \leq \sum_{k=\ell}^n \frac{\gamma(2^k)}{\gamma(2^\ell)} \prod_{\substack{j=\ell\\j\neq k}}^{n+1} \bigg| 1-\frac{\gamma(2^k)}{\gamma(2^j)}\bigg|^{-1} \frac{\gamma(2^k)}{\gamma(2^{n+1})} e^{-\frac{\gamma(2^k)}{4}t} \\
& \qquad\qquad\qquad + \frac{\gamma(2^{n+1})}{\gamma(2^\ell)} \prod_{j=\ell}^n \biggl|1-\frac{\gamma(2^{n+1})}{\gamma(2^j)}\biggr|^{-1} e^{-\frac{\gamma(2^{n+1})}{4}t} \\
& \leq \frac32\sum_{k=\ell}^{n+1} \frac{\gamma(2^k)}{\gamma(2^\ell)} \prod_{\substack{j=\ell\\j\neq k}}^{n+1} \bigg| 1-\frac{\gamma(2^k)}{\gamma(2^j)}\bigg|^{-1} 2^{-\beta(n+1-k)}e^{-\frac{\gamma(2^k)}{4}t} \\
& \leq \frac32 c \sum_{k=\ell}^{n+1} 4^{\beta(k-\ell)} 2^{-\frac{\beta}{2}(k-\ell)^2} 2^{-\beta(n+1-k)}e^{-\frac{\gamma(2^k)}{4}t}.
\end{split}
\end{equation*}
From this estimate it is easily seen that \eqref{linearfund2} follows (for a possibly larger constant $c_1>0$). The existence of the limit $\Psi_\infty^{(\ell)}(t)$ is an immediate consequence of \eqref{linearfund2}, which also implies \eqref{linearfund3} (taking a larger $c_1$) by writing
\begin{equation*}
\big| \Psi_n^{(\ell)}(t) - \Psi_{\infty}^{(\ell)}(t) \big|  \leq \sum_{m=n}^\infty \big| \Psi_m^{(\ell)}(t) - \Psi_{m+1}^{(\ell)}(t) \big| 
\leq c_12^{\beta\ell}e^{-\frac{\gamma(2^\ell)}{4}t}\sum_{m=n}^\infty 2^{-\beta m} \,.
\end{equation*}

It only remains to show \eqref{linearfund4}. To this aim, notice that we have the explicit formula
\begin{equation} \label{plinear16}
\Psi_\infty^{(\ell)}(t) := \sum_{k=\ell}^\infty \frac{\gamma(2^k)}{\gamma(2^\ell)} \prod_{\substack{j=\ell\\j\neq k}}^\infty \biggl(1-\frac{\gamma(2^k)}{\gamma(2^j)}\biggr)^{-1} e^{-\frac{\gamma(2^k)}{4}t}
\end{equation}
(the series is absolutely convergent in view of \eqref{plinear13}, which also implies that we can pass to the limit as $n\to\infty$ in \eqref{plinear12}).
Furthermore we have for all $n\geq\ell$
\begin{align} \label{plinear15}
2^{\beta n}\bigl( \Psi_n^{(\ell)}(t) - \Psi_\infty^{(\ell)}(t) \bigr)
& = 2^{\beta n}\sum_{k=\ell}^n \frac{\gamma(2^k)}{\gamma(2^\ell)} \prod_{\substack{j=\ell\\j\neq k}}^\infty \biggl( 1-\frac{\gamma(2^k)}{\gamma(2^j)}\biggr)^{-1} \Biggl[ \prod_{j=n+1}^\infty \biggl(1-\frac{\gamma(2^k)}{\gamma(2^j)}\biggr) - 1 \Biggr] e^{-\frac{\gamma(2^k)}{4}t} \nonumber \\
& \qquad + 2^{\beta n} \sum_{k=n+1}^\infty \frac{\gamma(2^k)}{\gamma(2^\ell)} \prod_{\substack{j=\ell\\j\neq k}}^\infty \biggl( 1-\frac{\gamma(2^k)}{\gamma(2^j)}\biggr)^{-1} e^{-\frac{\gamma(2^k)}{4}t} \,.
\end{align}
We now want to show that the previous expression has a limit as $n\to\infty$. Notice first that the last term in \eqref{plinear15} vanishes as $n\to\infty$, since in view of \eqref{plinear13}
\begin{equation} \label{plinear17}
2^{\beta n} \sum_{k=n+1}^\infty \frac{\gamma(2^k)}{\gamma(2^\ell)} \prod_{\substack{j=\ell\\j\neq k}}^\infty \bigg| 1-\frac{\gamma(2^k)}{\gamma(2^j)}\bigg|^{-1} e^{-\frac{\gamma(2^k)}{4}t}
\leq c 2^{\beta n}e^{-\frac{\gamma(2^{n+1})}{4}t} \sum_{k=n+1}^\infty 4^{\beta(k-\ell)} 2^{-\frac{\beta}{2}(k-\ell)^2}
\to 0 \,.
\end{equation}
For the first term on the right-hand side of \eqref{plinear15}, we first compute, using \eqref{kernel5},
\begin{equation} \label{plinear18}
\begin{split}
\lim_{n\to\infty} 2^{\beta n}\Biggl[ \prod_{j=n+1}^\infty \biggl(1-\frac{\gamma(2^k)}{\gamma(2^j)}\biggr) &- 1 \Biggr]
= \lim_{n\to\infty} 2^{\beta n} \Biggl[ \exp\Biggl(\sum_{j=n+1}^\infty\ln\Bigl(1-\frac{\gamma(2^k)}{\gamma(2^j)}\Bigr) \Biggr) - 1 \Biggr] \\
& = \lim_{n\to\infty} 2^{\beta n} \Biggl[ -\sum_{j=n+1}^\infty \frac{\gamma(2^k)}{\gamma(2^j)} + O\bigl(2^{2\beta(k-n)}\bigr) \Biggr]
= -\frac{\gamma(2^k)}{2^{\beta}-1}.
\end{split}
\end{equation}
It follows from \eqref{plinear15}, \eqref{plinear17} and \eqref{plinear18} that
\begin{equation} \label{plinear19}
\begin{split}
\lim_{n\to\infty} 2^{\beta n}\bigl( \Psi_n^{(\ell)}(t) - \Psi_\infty^{(\ell)}(t) \bigr)
= - \frac{1}{2^{\beta}-1} \sum_{k=\ell}^\infty \frac{(\gamma(2^k))^2}{\gamma(2^\ell)} \prod_{\substack{j=\ell\\j\neq k}}^\infty \biggl( 1-\frac{\gamma(2^k)}{\gamma(2^j)}\biggr)^{-1} e^{-\frac{\gamma(2^k)}{4}t} \,.
\end{split}
\end{equation}
The last estimate in \eqref{linearfund4} follows directly from \eqref{linearfund3}.
\end{proof}

By means of the fundamental solutions constructed in Lemma~\ref{lem:linearfundsol} we can now analyze the behaviour of solutions to \eqref{linear1} for $n\to\infty$. Recalling the notation \eqref{linear10} for the semigroup generated by the linear equation \eqref{linear1}, we have the following result.

\begin{lemma}\label{lem:linear2}
Let $\theta$, $\tilde{\theta}$ be fixed parameters satisfying the assumption \eqref{theta}, and let $y^0\in\mathcal{Y}_\theta$ be a given initial datum.
Then for every sufficiently large $n_0\in\N$ there exists a constant $\overline{C}_{n_0}>0$, depending on $M$, $\theta$, $\tilde{\theta}$, and $n_0$, such that the solution $S(t)(y^0)$ to the linear problem \eqref{linear1} with initial datum $y^0$, constructed in Lemma~\ref{lem:linear1}, satisfies the estimate
\begin{equation} \label{linear8}
2^{\tilde{\theta}n}|S_n(t)(y^0)-S_{n+1}(t)(y^0)| \leq \overline{C}_{n_0}\|y^0\|_\theta (1+t^{-\frac{\tilde{\theta}-\theta}{\beta}}) e^{-\nu t} \quad\text{for all $n> n_0$ and $t>0$,}
\end{equation}
where $\nu$ is as in Lemma~\ref{lem:linear1}. In particular for every $t>0$ is well-defined the limit
\begin{equation} \label{linear9}
S_\infty(t)(y^0) := \lim_{n\to\infty} S_n(t)(y^0) .
\end{equation}
Furthermore, there exists the limit
\begin{equation}  \label{linear11}
D^\beta_\infty S(t)(y^0) := \lim_{n\to\infty} 2^{\beta n}\bigl( S_n(t)(y^0) - S_\infty(t)(y^0) \bigr) \,.
\end{equation}
\end{lemma}

\begin{proof}
By means of the fundamental solutions $\Psi_n^{(\ell)}$ we can write a representation formula for the solution to \eqref{linear1} in the region $n>n_0$, where $n_0\in\N$ is to be chosen sufficiently large, in terms of the initial values $y_n^0$ and of the values of the solution for $n=n_0$. More precisely, we solve the initial/boundary value problem
\begin{equation} \label{plinear21}
\begin{cases}
\frac{\de y_n}{\de t} = \frac{\gamma(2^{n})}{4} \bigl( y_{n-1}-y_n \bigr) + r_n(t) & n>n_0, \\
y_n(0) = y_n^0 & n>n_0,\\
y_{n_0}(t) = \lambda(t) & t>0,
\end{cases}
\end{equation}
for given functions $\lambda(t)$ and $r_n(t)$. Notice that, by Lemma~\ref{lem:linear1}, we have the estimate
\begin{equation} \label{plinear22}
|\lambda(t)-\bar{m}| = |y_{n_0}(t)-\bar{m}| \leq 2C_{n_0}\|y^0\|_{\theta}e^{-\nu t},
\end{equation}
where $\bar{m}$ is the constant introduced in \eqref{linear6}. Moreover in view of \eqref{kernel5} and \eqref{linear3} we can assume that
\begin{equation} \label{plinear23}
|r_n(t)|:= \Big|-\frac{\gamma(2^n)}{4}\sigma_n\bigl(y_n(t)-y_{n+1}(t)\bigr)\Big| \leq c_2 2^{\beta n}e^{-A_M2^n}|y_n(t)-y_{n+1}(t)|
\end{equation}
for a uniform constant $c_2>0$.

By Duhamel's Principle we can write the solution to \eqref{plinear21}, for all $n>n_0$, as
\begin{equation} \label{plinear24}
\begin{split}
y_n(t)
& = \frac{\gamma(2^{n_0+1})}{4}\int_0^t \Psi_n^{(n_0+1)}(t-s)\lambda(s)\de s + \sum_{\ell=n_0+1}^n \Psi_n^{(\ell)}(t)y_\ell^0 \\
& \qquad\qquad + \int_0^t \sum_{\ell=n_0+1}^n \Psi_n^{(\ell)}(t-s)r_\ell(s)\de s \,.
\end{split}
\end{equation}
Then, in view of the identity \eqref{plinear12bis}, we estimate the difference between $y_n(t)$ and $y_{n+1}(t)$, $n\geq n_0+1$, as follows:
\begin{equation*}
\begin{split}
|y_n(t) - y_{n+1}(t)|
& \leq \frac{\gamma(2^{n_0+1})}{4} \int_0^t \big| \Psi_n^{(n_0+1)}-\Psi_{n+1}^{(n_0+1)} \big|(t-s) |\lambda(s)-\bar{m}| \de s \\
& \qquad+ |\bar{m}|\frac{\gamma(2^{n_0+1})}{4}\int_t^\infty \big|\Psi_n^{(n_0+1)}(s)-\Psi_{n+1}^{(n_0+1)}(s)\big|\de s  \\
& \qquad+ \sum_{\ell=n_0+1}^{n} \big| \Psi_n^{(\ell)}(t)-\Psi_{n+1}^{(\ell)}(t) \big| |y_\ell^0| + \Psi_{n+1}^{(n+1)}(t)|y_{n+1}^0| \\
& \qquad+ \int_0^t\sum_{\ell=n_0+1}^{n+1}  \big| \Psi_n^{(\ell)}-\Psi_{n+1}^{(\ell)} \big|(t-s) |r_\ell(s)|\de s \,.
\end{split}
\end{equation*}
Using the estimates \eqref{linearfund2}, \eqref{plinear22}, \eqref{linear6bis}, and \eqref{plinear23} in the previous inequality we obtain
\begin{align} \label{plinear25}
2^{\tilde{\theta}n}|y_n(t) - y_{n+1}(t)|
& \leq 2c_1C_{n_0}\|y^0\|_\theta\frac{\gamma(2^{n_0+1})}{4}2^{\beta(n_0+1)}2^{(\tilde{\theta}-\beta)n} \int_0^t e^{-\frac{\gamma(2^{n_0+1})}{4}(t-s)} e^{-\nu s} \de s \nonumber \\
& \qquad + c_1C_0\|y^0\|_\theta\frac{\gamma(2^{n_0+1})}{4} 2^{\beta(n_0+1)}2^{(\tilde{\theta}-\beta)n} \int_t^\infty e^{-\frac{\gamma(2^{n_0+1})}{4}s}\de s \nonumber \\
& \qquad + c_1 \|y^0\|_\theta 2^{(\tilde{\theta}-\beta)n} \sum_{\ell=n_0+1}^{n+1} 2^{(\beta-\theta)\ell}e^{-\frac{\gamma(2^\ell)}{4}t} \\
& \qquad + c_1c_2 2^{(\tilde{\theta}-\beta)n} \int_0^t\sum_{\ell=n_0+1}^{n+1} 2^{2\beta\ell}e^{-\frac{\gamma(2^\ell)}{4}(t-s)} e^{-A_M2^\ell} |y_\ell-y_{\ell+1}|(s) \de s \,. \nonumber
\end{align}
The term which requires more attention is the third one on the right-hand side of \eqref{plinear25}, which becomes singular as $t\to0^+$. We let $n_t:=\lfloor-\frac{1}{\beta}\frac{\ln t}{\ln2}\rfloor$, where $\lfloor\cdot\rfloor$ denotes the integer part, and $t_0:=2^{-\beta}$. Then for $t\leq t_0$ we have $n_t\geq1$ and, for $n\geq n_t$,
\begin{equation} \label{plinear26}
\begin{split}
2^{(\tilde{\theta}-\beta)n}\sum_{\ell=n_0+1}^{n+1} & 2^{(\beta-\theta)\ell}e^{-\frac{\gamma(2^\ell)}{4}t}
\leq 2^{(\tilde{\theta}-\beta)n} e^{-\frac{\gamma(2^{n_0+1})}{8}t} \sum_{\ell=0}^{n+1} 2^{(\beta-\theta)\ell}e^{-\frac{\gamma(2^\ell)}{8}t} \\
& \leq 2^{(\tilde{\theta}-\beta)n} e^{-\frac{\gamma(2^{n_0+1})}{8}t} \biggl(  \sum_{\ell=0}^{n_t-1} 2^{(\beta-\theta)\ell} + 2^{(\beta-\theta) n_t}\sum_{j=0}^\infty 2^{(\beta-\theta) j} e^{-\frac{\gamma(2^{j+n_t})}{8}t} \biggr)\\
& \leq 2^{(\tilde{\theta}-\beta)n} e^{-\frac{\gamma(2^{n_0+1})}{8}t} \biggl(\frac{2^{(\beta-\theta)n_t}}{2^{\beta-\theta}-1} + 2^{(\beta-\theta) n_t}\sum_{j=0}^\infty 2^{(\beta-\theta) j} e^{-\frac{2^{\beta(j-1)}}{16}} \biggr)\\
& \leq c_\theta 2^{(\tilde{\theta}-\beta)n}2^{(\beta-\theta)n_t} e^{-\frac{\gamma(2^{n_0+1})}{8}t}
\leq c_\theta 2^{(\tilde{\theta}-\theta)n_t}e^{-\frac{\gamma(2^{n_0+1})}{8}t} \\
& \leq c_\theta t^{-\frac{\tilde{\theta}-\theta}{\beta}}e^{-\frac{\gamma(2^{n_0+1})}{8}t},
\end{split}
\end{equation}
where we used the bound $\gamma(2^{j+n_t})\geq \frac12 2^{\beta j}2^{\beta n_t}\geq\frac{1}{2t}2^{\beta(j-1)}$, and $c_\theta$ is a constant depending only on $\theta$. It is easily seen that the same estimate holds for $n<n_t$. For values of $t>t_0$, it is straightforward to obtain the bound
\begin{equation}  \label{plinear26bis}
2^{(\tilde{\theta}-\beta)n}\sum_{\ell=n_0+1}^{n+1} 2^{(\beta-\theta)\ell}e^{-\frac{\gamma(2^\ell)}{4}t} \leq c_\theta e^{-\frac{\gamma(2^{n_0+1})}{8}t} 
\end{equation}
with a possibly larger constant $c_\theta$. Hence, combining \eqref{plinear26} and \eqref{plinear26bis} we find
\begin{equation}  \label{plinear26ter}
2^{(\tilde{\theta}-\beta)n}\sum_{\ell=n_0+1}^{n+1} 2^{(\beta-\theta)\ell}e^{-\frac{\gamma(2^\ell)}{4}t}
\leq c_\theta (1+t^{-\frac{\tilde{\theta}-\theta}{\beta}})e^{-\frac{\gamma(2^{n_0+1})}{8}t} .
\end{equation}

Now, setting $w(t):=\sup_{n>n_0} 2^{\tilde{\theta} n}|y_n(t)-y_{n+1}(t)|$ and $L:=\frac{\gamma(2^{n_0+1})}{8}$, by inserting \eqref{plinear26ter} into \eqref{plinear25} we can conclude that there exist constants $C>0$ (depending on $M$, $\theta$, $n_0$) and $c_3>0$ (depending only on $M$) such that
\begin{equation*}
w(t) \leq C\|y^0\|_\theta \biggl( \int_0^t e^{-L(t-s)}e^{-\nu s}\de s + \int_t^\infty e^{-Ls}\de s + \bigl(1+t^{-\frac{\tilde{\theta}-\theta}{\beta}}\bigr)e^{-Lt} \biggr) + c_3\int_0^t e^{-L(t-s)}w(s)\de s \,,
\end{equation*}
which yields, for a possibly larger constant $C$,
\begin{equation} \label{plinear27}
w(t) \leq C\|y^0\|_\theta(1+t^{-\frac{\tilde{\theta}-\theta}{\beta}})e^{-\nu t} + c_3\int_0^t e^{-L(t-s)}w(s)\de s \,.
\end{equation}
We can then apply a Gr\"onwall-type argument to obtain an exponential-in-time decay of $w(t)$: letting $W(t):=\int_0^t e^{-L(t-s)}w(s)\de s$, \eqref{plinear27} yields
\begin{equation*}
\frac{\de W}{\de t} = w(t) - LW(t) \leq C\|y^0\|_\theta(1+t^{-\frac{\tilde{\theta}-\theta}{\beta}})e^{-\nu t} + (c_3-L)W(t) .
\end{equation*}
Recalling the definition of $L$, we can choose $n_0$ sufficiently large so that $c_3-L<-\nu$; then from the previous differential inequality we obtain $W(t) \leq C\|y^0\|_\theta e^{-\nu t}$ (for a larger constant $C$, depending on $M$, $n_0$, $\theta$, $\tilde{\theta}$), and in turn by \eqref{plinear27}
\begin{equation*}
w(t) \leq C\|y^0\|_\theta(1+t^{-\frac{\tilde{\theta}-\theta}{\beta}})e^{-\nu t} .
\end{equation*}
This completes the proof of \eqref{linear8}, which in particular yields the existence of the limit $y_\infty(t):=\lim_{n\to\infty}y_n(t)$.

It remains to prove the existence of the limit in \eqref{linear11}. Notice that by Lebesgue's Dominated Convergence Theorem and the estimate \eqref{linearfund3} we can pass to the limit as $n\to\infty$ in \eqref{plinear24}:
\begin{equation} \label{plinear28}
\begin{split}
y_\infty(t)
& = \frac{\gamma(2^{n_0+1})}{4}\int_0^t \Psi_\infty^{(n_0+1)}(t-s)\lambda(s)\de s + \sum_{\ell=n_0+1}^\infty \Psi_\infty^{(\ell)}(t)y_\ell^0 \\
& \qquad\qquad + \int_0^t \sum_{\ell=n_0+1}^\infty \Psi_\infty^{(\ell)}(t-s)r_\ell(s)\de s \,.
\end{split}
\end{equation}
Then, using the expressions \eqref{plinear24} and \eqref{plinear28}, and recalling \eqref{linearfund4}, one can show that
\begin{equation*}
\begin{split}
D^\beta_\infty y(t)
& := \lim_{n\to\infty} 2^{\beta n}\bigl( y_n(t) - y_\infty(t) \bigr) \\
& = \frac{\gamma(2^{n_0+1})}{4}\int_0^t D^\beta_\infty\Psi^{(n_0+1)}(t-s)\lambda(s)\de s + \sum_{\ell=n_0+1}^\infty D^\beta_\infty\Psi^{(\ell)}(t)y_\ell^0 \\
& \qquad\qquad + \int_0^t \sum_{\ell=n_0+1}^\infty D^\beta_\infty\Psi^{(\ell)}(t-s)r_\ell(s)\de s \,,
\end{split}
\end{equation*}
where the uniform estimate \eqref{linearfund3} allows to pass to the limit under the integral sign.
\end{proof}

The proof of the result in Section~\ref{sect:linear} follows now by combining Lemma~\ref{lem:linear1} and Lemma~\ref{lem:linear2}.

\begin{proof}[Proof of Theorem~\ref{thm:linear}]
The first estimate \eqref{linear13b} follows directly by combining Lemma~\ref{lem:linear1} and Lemma~\ref{lem:linear2}, and yields the existence of the limit $S_\infty(t)(y^0)$. Moreover, another application of Lemma~\ref{lem:linear2} gives for all $m>n> n_0$
\begin{equation*}
\begin{split}
|S_n(t)(y^0)-S_m(t)(y^0)|
& \leq \sum_{k=n}^{m-1}|S_k(t)(y^0)-S_{k+1}(t)(y^0)| \\
& \leq \overline{C}_{n_0}\|y^0\|_\theta \bigl(1+t^{-\frac{\tilde{\theta}-\theta}{\beta}}\bigr)e^{-\nu t} \sum_{k=n}^{m-1}2^{-\tilde{\theta}k},
\end{split}
\end{equation*}
hence by passing to the limit as $m\to\infty$
\begin{equation} \label{linear14}
2^{\tilde{\theta}n}|S_n(t)(y^0)-S_\infty(t)(y^0)| \leq \overline{C}_{n_0}\frac{2^{\tilde{\theta}}}{2^{\tilde{\theta}}-1}\|y^0\|_\theta \bigl(1+t^{-\frac{\tilde{\theta}-\theta}{\beta}}\bigr)e^{-\nu t}
\end{equation}
provided that $\tilde{\theta}>0$. In particular, by using Lemma~\ref{lem:linear1} we also have
\begin{equation*}
\begin{split}
|S_\infty(t)(y^0)-\bar{m}|
& \leq |S_\infty(t)(y^0)-S_{n_0}(t)(y^0)| + |S_{n_0}(t)(y^0)-\bar{m}| \\
& \leq \biggl[\frac{\overline{C}_{n_0}2^{\tilde{\theta}}}{2^{\tilde{\theta}}-1} 2^{-\tilde{\theta} n_0} + C_{n_0}(2^{-n_0}+1) \biggr] \|y^0\|_\theta\bigl(1+t^{-\frac{\tilde{\theta}-\theta}{\beta}}\bigr)e^{-\nu t} \,.
\end{split}
\end{equation*}
Then \eqref{linear13} is a straightforward consequence of this estimate, \eqref{linear5}, and \eqref{linear14}. The identity \eqref{linear12} can be obtained by recalling the asymptotics \eqref{kernel5} and the fast decay rate \eqref{linear3} of $\sigma_n$ as $n\to\infty$.
\end{proof}

\begin{remark}
From the proof of Lemma~\ref{lem:linear2}, one can see that the constant in the estimate \eqref{linear13b} blows up if $\theta\to\beta$ or $\tilde{\theta}-\theta\to\beta$. From \eqref{linear14}, the constant in \eqref{linear13} explodes also if $\tilde{\theta}\to0$.
\end{remark}


\bigskip
\noindent
{\bf Acknowledgments.}
The authors acknowledge support through the CRC 1060 \textit{The mathematics of emergent effects} at the University of Bonn that is funded through the German Science Foundation (DFG).

\bibliographystyle{siam}
\bibliography{Bibliography}

\end{document}